\documentclass[11pt]{article}

\usepackage[colorlinks=true,citecolor=black,linkcolor=black,urlcolor=blue]{hyperref}
\usepackage[italian, english]{babel}
\usepackage[utf8]{inputenc}
\usepackage{latexsym}
\usepackage{booktabs} 
\usepackage{amsfonts}
\usepackage[pdftex]{graphicx}
\usepackage{subfig}
\usepackage{caption}
\usepackage{textcomp} 
\usepackage{amsmath,amssymb,amsthm}
\usepackage{tikz,color,graphicx}
\usepackage{pgfplots}
\usepackage{floatflt,epsfig}
\usepackage{bbm}

\usepackage{tikz}
\usetikzlibrary{positioning,chains,fit,shapes,calc}
\usetikzlibrary{arrows,intersections}

\theoremstyle{plain}
\newtheorem{teorema}{Theorem}[section]
\theoremstyle{definition}
\newtheorem{definizione}[teorema]{Definition}
\theoremstyle{definition}

\theoremstyle{definition}

\theoremstyle{definition}

\theoremstyle{plain}
\newtheorem{corollario}[teorema]{Corollary}
\theoremstyle{plain}
\newtheorem{proposizione}[teorema]{Proposition}
\theoremstyle{plain}
\newtheorem{lemma}[teorema]{Lemma}
\theoremstyle{plain}

\theoremstyle{definition}
\newtheorem{remark}[teorema]{Remark}
\theoremstyle{plain}

\numberwithin{equation}{section}
\numberwithin{teorema}{section}

\def\dfrac#1#2{\lower0.15ex\hbox{\large$\frac{#1}{#2}$}}

\begin{document}

\pagestyle{plain}
\pagenumbering{arabic}


\title{Wireless random-access networks \\ with bipartite interference graphs}

\author{\renewcommand{\thefootnote}{\arabic{footnote}}
\renewcommand{\thefootnote}{\arabic{footnote}}
Sem C.\ Borst\,%
\footnotemark[1]
\\
\renewcommand{\thefootnote}{\arabic{footnote}}
Frank den Hollander\,%
\footnotemark[2]
\\
\renewcommand{\thefootnote}{\arabic{footnote}}
Francesca R.\ Nardi\,%
\footnotemark[1]\,\,\,\,\footnotemark[3]
\\
\renewcommand{\thefootnote}{\arabic{footnote}}
Matteo Sfragara\,%
\footnotemark[2]\,\,\,\,\footnotemark[4]
}

\footnotetext[1]{%
Department of Mathematics and Computer Science, Eindhoven University of Technology, The Netherlands
}

\footnotetext[2]{%
Mathematical Institute, Leiden University, The Netherlands
}

\footnotetext[3]{%
Department of Mathematics, University of Florence, Italy
}

\footnotetext[4]{%
Department of Mathematics, Stockholm University, Sweden
}

\date{\today}

\maketitle

\begin{abstract}

We consider random-access networks where nodes represent servers with a queue and can be either active or inactive. A node deactivates at unit rate, while it activates at a rate that depends on its queue length, provided none of its neighbors is active. 
We consider arbitrary bipartite graphs in the limit as the initial queue lengths
become large and identify the transition time between the two states where one half of the network is active and the other half is inactive. The transition path is decomposed into a succession of transitions on complete bipartite subgraphs. We formulate a randomized greedy algorithm that takes the graph as input and gives as output the set of transition paths the network is most likely to follow. Along each path we determine the mean transition time and its law on the scale of its mean. Depending on the activation rates, we identify three regimes of behavior.

\medskip\noindent
\emph{Keywords:} Random-access networks, activation protocols, bipartite interference graphs, transition time, randomized algorithm. \\
\emph{MSC2010:} 
60K25, 
60K30, 
90B15, 
90B18. 
\\
\emph{Acknowledgments:} The research in this paper was supported through NWO Gravitation Grant 024.002.003--NETWORKS.
\end{abstract}

\newpage

\tableofcontents

\newpage


\section{Introduction}
\label{s:intro}

The present paper is a continuation of \cite{BdHNS18}. In Section~\ref{ss:motivation} we give our motivation, which is a summary of the more extensive motivation provided in \cite[Section 1.1]{BdHNS18}, where also \emph{relevant references} to the literature are included. In Section~\ref{ss:model} we formulate the random-access model whose performance we analyze in detail. In Section~\ref{ss:interference} we introduce the interference graph and recall a key theorem from \cite{BdHNS18} for the total transition time on complete bipartite graphs. In Section~\ref{ss:outline} we hint at the key idea behind our analysis, which involves transitions along a sequence of complete bipartite subgraphs selected via a randomized greedy algorithm, and give an outline of the remainder of the paper. 
  

\subsection{Motivation and background}
\label{ss:motivation}

We are interested in transition time asymptotics of \emph{queue-based random-access protocols in wireless networks}. Specifically, we consider a stylised stochastic model for a wireless network, represented in terms of an undirected graph $G = (S,E)$, referred to as the \emph{interference graph}. The set of nodes $S$ labels the servers and the set of edges $E$ indicates which pairs of servers interfere and are therefore prevented from simultaneous activity (see Fig.~\ref{fig:network}). We denote by $X(t)=(X_w(t))_{w \in S}$ the joint activity state at time $t$, which is an element of the state space 
\begin{equation}
\label{Xdef}
\mathcal{X} = \big\{x \in \{0,1\}^{|S|}\colon\, x_w x_{\bar{w}} = 0\,\,\,\forall\, (w,\bar{w}) \in E\big\},
\end{equation}
where $x_w = 0$ means that node $w$ is inactive and $x_w =1$ means that node $w$ is active. 

\begin{figure}[htbp]
\begin{center}
\vspace{0.5cm}
\includegraphics[width=.45\linewidth]{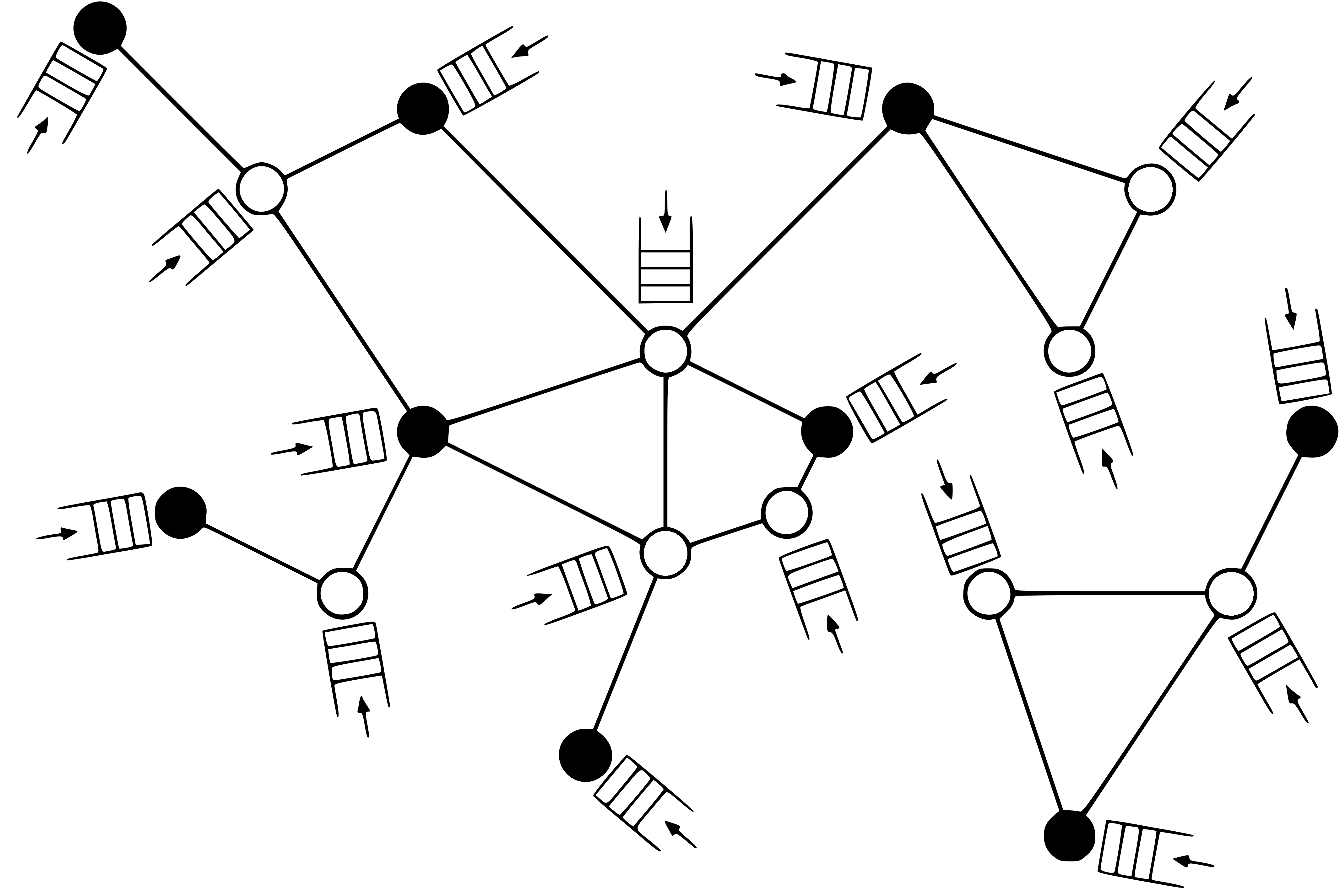}
\vspace{0.5cm}
\caption{\small A random-access network. Each node represents a server with a queue. 
Packets arrive that require a random service time.}
\label{fig:network}
\end{center}
\vspace{-0.5cm}
\end{figure}

We assume that packets arrive at the nodes as independent Poisson processes and have independent exponentially distributed sizes. When a packet arrives at a node, it joins the queue at that node and the queue length undergoes an instantaneous jump equal to the size of the arriving packet. The queue length decreases at a constant speed $c$ (as long as it is positive) when the node is active. We denote by $Q(t)=(Q_w(t))_{w \in S}$ the joint queue length state at time $t$. When node $w$ is inactive at time $t$, it activates at a rate that is an increasing function of $Q_w(t)$, \emph{provided none of its neighbors is active}. When a node is active at time $t$, it deactivates at rate $1$. The joint process
\begin{equation}
\label{MP}
(X(t),Q(t))_{t \geq 0}
\end{equation} 
evolves as a time-homogeneous Markov process with state space $\mathcal{X} \times \mathbb{R}_{\geq 0}^{|S|}$, since the transition rates depend on time only via the the current state of the vector.

The Markov process in \eqref{MP} may be viewed as a \emph{hard-core interaction model with state-dependent activation rates}. Its present state not only depends on the history of the packet arrivals and their service times (which cause upward jumps in the queue lengths), but also on the history of the activity process (through the gradual reduction in queue lengths during activity periods). The state-dependent nature of the activation rates raises interesting and challenging issues from a methodological perspective. We are particularly interested in what happens when the initial queue lengths 
\begin{equation}
Q(0) = (Q_w(0))_{w \in S},
\end{equation} 
become large. In this limit the network exhibits \emph{metastable behavior}: before becoming active, an inactive node must wait until all the nodes it interferes with have become inactive simultaneously, which takes a long time when the queues at these nodes are long and the activation rates grow without bound as function of the queue length.

In \cite{BdHNS18} we focused on the simple case of a \emph{complete bipartite} interference graph: the node set can be partitioned into two nonempty sets $U$ and $V$ such that two nodes interfere if and only if one belongs to $U$ and the other belongs to $V$. In the present paper we turn our attention to \emph{general bipartite} interference graphs, for which not necessarily all nodes in $U$ interfere with all nodes in $V$. This case will turn out to be considerably more challenging. We will be interested in starting from the state where all the nodes in $U$ are active and all the nodes in $V$ are inactive, and examining the transition time to the state where all the nodes in $U$ are inactive and all the nodes in $V$ are active. We refer to this transition as a \emph{metastable crossover}. It will turn out that, in order to achieve the full transition, the network goes through a succession of \emph{subtransitions}, in which a certain succession of complete bipartite subgraphs achieve a metastable crossover and, in doing so, effectively remove themselves from the network. This succession depends in a delicate manner on the \emph{full structure} of the bipartite interference graph, which we capture with the help of a \emph{randomized greedy algorithm} that identifies which subtransition occurs first, which second, etc., and with what probability. By combining the results in \cite{BdHNS18} with a detailed analysis of the algorithm, we are able to determine the distribution of the \emph{full metastable crossover time} to leading order as the initial queue lengths become large.  


\subsection{Mathematical model}
\label{ss:model}

We consider the bipartite graph $G=((U,V),E)$, where $U \cup V$ is the set of nodes and $E$ is the set of edges that connect a node in $U$ to a node in $V$, and vice versa (edges are undirected). We set $N = |V|$. We recall some definitions and basic facts from \cite{BdHNS18}.

\begin{definizione}{\bf [Key notions 1]} \label{defcomplete}
$\mbox{}$\\
{\bf (1) State of a node.}
A node in the network can be either \textit{active} or \textit{inactive}. The state of node $w$ at time $t$ is described by a Bernoulli random variable $X_w(t) \in \{0,1\}$, defined as
\begin{equation} 
X_w(t) = 
\begin{cases} 
0, \text{ if } w \text{ is inactive at time } t ,\\ 
1, \text{ if } w \text{ is active at time } t. 
\end{cases}
\end{equation} 
The configuration at time $t$ is denoted by
\begin{equation}
X(t) = \{ X_w(t) \}_{w \in U \cup V}.
\end{equation}
We denote by $1_U$ ($1_V$) the configuration where all nodes in $U$ are active (inactive) and all nodes in $V$ are inactive (active).\\
{\bf (2) Transition time.} 
Our main object of interest is the {\it transition time} to $1_V$ starting from $1_U$, i.e.,
\begin{equation} \label{transition}
\tau_{1_V} = \min \big\{t \geq 0\colon\, X(t) = 1_V\} \quad \text{ given } \quad X(0) = 1_U.
\end{equation}
{\bf (3) Activation and deactivation of a node.}
An active node $w$ turns inactive according to a {\it deactivation Poisson clock:} when the clock ticks the node switches itself off. Conversely, an inactive node $w$ attempts to become active according to an {\it activation Poisson clock}, but the attempt is successful only when no neighbors of $i$ are active. We are interested in what are called \textit{internal models}, where the activation rate at node $w$ at time $t$ depends on the queue length at node $w$ at time $t$. The deactivation rate is $1$ and does not depend on the queue length.\\ 
{\bf (4) Queue length at a node.}
Let $t \mapsto Q_w^+(t)$ be the \textit{input process} describing packets arriving at node $w$ according to a Poisson process $t \mapsto N_w(t) = \mathrm{Poisson}(\lambda t)$ and requiring i.i.d.\ exponential service times $Y_{wn}$, $n \in \mathbb{N}$, with rate $\mu_U$ for $w \in U$ and $\mu_V$ for $w \in V$. This is a compound Poisson process with mean $\rho_U = \lambda / \mu_U$ for $w \in U$ and $\rho_V = \lambda / \mu_V$ for $w \in V$. Let $t \mapsto Q_w^-(t)$ be the \textit{output process} representing the cumulative amount of work that is processed by the server at node $w$ in the time interval $[0,t]$ at rate $c$, which equals $cT_w(t) = c \int_0^t X_w(s) ds$. In order to ensure that the queue length tends to decrease when a node is active, we assume that $\rho_U < c$ and $\rho_V < c$. Define
\begin{equation} 
\Delta_w(t) = Q_w^+(t) - Q_w^-(t) = \sum_{n=0}^{N_w(t)} Y_{wn} - c T_w(t)
\end{equation}
and let $s^* = s^*(t)$ be the value where $\sup_{s \in [0,t]} [\Delta_w(t) - \Delta_w(s)]$ is reached, i.e., equals $[\Delta_w(t) - \Delta_w(s^*-)]$. Let $Q_w(t) \in \mathbb{R}_{\geq 0}$ denote the queue length at node $w$ at time $t$. Then
\begin{equation}
Q_w(t) = \max\big\{ Q_w(0) + \Delta_w(t),\,\Delta_w(t)-\Delta_w(s^*-) \big\},
\end{equation}
where $Q_w(0)$ is the initial queue length. The maximum is achieved by the first term when $Q_w(0) \geq -\Delta_w(s^*-)$ (the queue length never sojourns at $0$), and by the second term when $Q_w(0) < -\Delta_w(s^*-)$ (the queue length sojourns at $0$ at time $s^*-$). \\
{\bf (5) Initial queue length.}
The \textit{initial queue length} is assumed to be given by
\begin{equation}
\label{initialqueues}
Q_w(0) = 
\left\{\begin{array}{ll} 
\gamma_U r, &w \in U, \\  
\gamma_V r, &w \in V, 
\end{array}
\right.
\end{equation}
where $\gamma_U, \gamma_V > 0$, and $r$ is a parameter that tends to infinity. In order to ensure that the queue lengths at nodes in $V$ always remain of order $r$, we assume that 
\begin{equation}
\label{assumptionparameters}
\frac{\gamma_U}{c- \rho_U} < \frac{\gamma_V}{c-\rho_V}.
\end{equation}
We will often write $Q_U(0)$ and $Q_V(0)$ to indicate the initial queue lengths at nodes in $U$ and $V$, respectively.\\
{\bf (6) Dependence of activation rate on queue length.} 
Let $g_U, g_V \in \mathcal{G}$ with
\begin{equation}
\mathcal{G} = \Big\{g\colon\,\mathbb{R}_{\geq 0} \to \mathbb{R}_{\geq 0}\colon\,
g \text{ non-decreasing and continuous},\, g(0)=0,\, \lim_{x \to \infty} g(x) = \infty\Big\}.
\end{equation}
The deactivation clocks tick at rate $1$, while the activation clocks tick at rate 
\begin{equation}
\label{rint}
r_w(t) = 
\left\{\begin{array}{ll} 
g_U(Q_w(t)), &w \in U, \\ 
g_V(Q_w(t)), &w \in V,
\end{array}
\right. 
\qquad t \geq  0.
\end{equation}
We focus on the particular choice
\begin{equation} \label{aggress1}
\begin{array}{ll}
g_U(x) = B x^{\beta}, &x \in [0, \infty), \\
g_V(x) = B' x^{\beta'}, &x \in [0, \infty),
\end{array}
\end{equation}
with $B,B', \beta, \beta' \in (0, \infty)$. We assume that nodes in $V$ are {\it much more aggressive} than nodes in $U$, namely,
\begin{equation} \label{beta'}
\beta' > \beta +1.
\end{equation}
As we will see later, this ensures that the transition path from $1_U$ to $1_V$ can be decomposed into a succession of transitions on complete bipartite subgraphs. 
\end{definizione}


\subsection{Interference graph}
\label{ss:interference}

Write $\mathbb{P}_{1_U}$ and $\mathbb{E}_{1_U}$ to denote probability and expectation on path space given that the initial configuration is $1_U$ and the initial queue lengths are as in \eqref{initialqueues}. We say that an event occurs \textit{with high probability} if its $\mathbb{P}_{1_U}$-probability tends to 1 as $r \to \infty$.

In \cite[Theorem 1.7]{BdHNS18}, building on results from \cite{BdHNT19}, we analyzed the mean transition time $\mathbb{E}_{1_U} [\tau_{1_V}]$ and the law of $\tau_{1_V}/\mathbb{E}_{1_U} [\tau_{1_V}]$ for the special case where the interference graph is a {\it complete bipartite graph}. They are strongly related to the initial queue lengths $Q_U(0)$ at the nodes in $U$.

\begin{figure}[htbp]
\vspace{0.4cm}
\begin{center}
\setlength{\unitlength}{0.3cm}
\begin{picture}(10,8)(18,0)
{\thicklines
\qbezier(0,-1)(0,3)(0,7)
\qbezier(-1,0)(6,0)(12,0)
\qbezier(0.1,5)(3,0.2)(10,0.2)
}
\put(12.5,-.25){$x$}
\put(-.8,7.6){$\mathcal{P}_{\mathrm{sub}}(x)$}
{\thicklines
\qbezier(17,-1)(17,3)(17,7)
\qbezier(16,0)(23,0)(29,0)
\qbezier(17.1,5)(19,0.2)(23,0)
\qbezier(23,0)(25,0)(27,0)
}
\qbezier[40](17.1,5)(21,5)(23,0)
\put(29.5,-.25){$x$}
\put(16.2,7.6){$\mathcal{P}_{\mathrm{cr}}(x)$}
\put(22.5,-1.7){$\tfrac{1}{C}$}
\put(23,0){\circle*{0.35}}
{\thicklines
\qbezier(34,-1)(34,3)(34,7)
\qbezier(33,0)(40,0)(46,0)
\qbezier(34.1,5)(37,5)(40,5)
\qbezier(40,5)(40,2.5)(40,0)
}
\put(46.5,-.25){$x$}
\put(33.2,7.6){$\mathcal{P}_{\mathrm{sup}}(x)$}
\put(39.7,-1.5){$1$}
\put(40,0){\circle*{0.35}}
\end{picture}
\end{center}
\vspace{0.5cm}
\caption{\small Trichotomy for $x \mapsto \mathcal{P}(x)$: $\beta \in (0, \frac{1}{|U|-1})$ (left);
$\beta = \frac{1}{|U|-1}$ (center); $\beta \in (\frac{1}{|U|-1},\infty)$ (right). The curve in the 
center is convex when $C \in (0,\tfrac12)$ and concave when $C \in (\tfrac12,1)$. The curve 
on the right is the limit of the curve in the center as $C \uparrow 1$.}
\label{fig:trichotomy}
\end{figure}
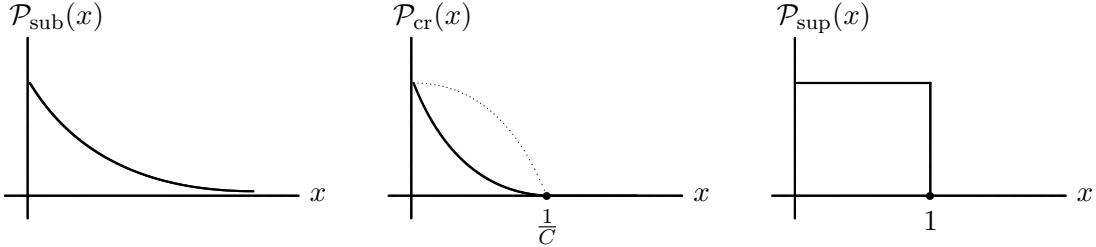

\begin{teorema}{\bf [Transition time for complete bipartite graph \cite[Theorem 1.7]{BdHNS18}]} \label{teoremapaper1} 
Let $G$ be a complete bipartite graph. Suppose that \eqref{aggress1}--\eqref{beta'} hold. Suppose that the initial queue lengths at the nodes in $U$ equal $Q_U(0) = \gamma_U r$.
\begin{itemize}
\item[{\rm (I)}] $\beta \in (0, \frac{1}{|U|-1})$: subcritical regime. The transition time satisfies
\begin{equation}
\mathbb{E}_{1_U} [\tau_{1_V}] = F_{\mathrm{sub}} \, Q_U(0)^{\beta(|U|-1)}\, [1+o(1)], \qquad r \to \infty,
\end{equation}
with $F_{\mathrm{sub}}= \frac{1}{|U|B^{-(|U|-1)}}$,
and
\begin{equation} 
\lim_{r \to \infty} \mathbb{P}_{1_U} \bigg( \frac{\tau_{1_V}} {\mathbb{E}_{1_U}[\tau_{1_V}]} > x \bigg) = \int_x^{\infty} \mathcal{P}_{\mathrm{sub}}(y)\,dy =  e^{-x}, \qquad x \in [0,\infty)
\end{equation}
with
\begin{equation}
\mathcal{P}_{\mathrm{sub}}(z) = e^{-z}, \qquad z \in [0,\infty).
\end{equation}
\item[{\rm (II)}] $\beta = \frac{1}{|U|-1}$: critical regime. The transition time satisfies
\begin{equation}
\mathbb{E}_{1_U} [\tau_{1_V}] = F_{\mathrm{cr}} \,Q_U(0)\, [1+o(1)], \qquad r \to \infty,
\end{equation}
with $F_{\mathrm{cr}} =\frac{1}{|U|B^{-(|U|-1)} + c-\rho_U}$,
and
\begin{equation} 
\begin{split}
\lim_{r \to \infty} \mathbb{P}_{1_U} \bigg( \frac{\tau_{1_V}} {\mathbb{E}_{1_U}[\tau_{1_V}]} > x \bigg) & = \int_x^{\infty} \mathcal{P}_{\mathrm{cr}}(y)\, dy \\
& = \left\{\begin{array}{ll}
(1-C x)^{\frac{1-C}{C}}, &\text{ if } x \in [0, \frac{1}{C}), \\[0.2cm]
0,  &\text{ if } x \in [\frac{1}{C}, \infty),
\end{array}
\right. \qquad x \in [0,\infty),
\end{split}
\end{equation}
with
\begin{equation}
\mathcal{P}_{\mathrm{cr}}(z) = \left\{\begin{array}{ll}
(1-C)(1-C z)^{\frac{1}{C}-2}, &\text{ if } z \in [0, \frac{1}{C}), \\[0.2cm]
0,  &\text{ if } z \in [\frac{1}{C}, \infty),
\end{array}
\right.
\end{equation}
and $C = F_{\mathrm{cr}}\, (c-\rho_U)  \in (0,1)$.

\item[{\rm (III)}] $\beta \in (\frac{1}{|U|-1},\infty)$: supercritical regime. The transition time satisfies
\begin{equation}
\mathbb{E}_{1_U} [\tau_{1_V}] = F_{\mathrm{sup}} \,Q_U(0) \, [1+o(1)], \qquad r \to \infty,
\end{equation}
with $F_{\mathrm{sup}}= \frac{1}{c-\rho_U}$,
and
\begin{equation} 
\lim_{r \to \infty} \mathbb{P}_{1_U} \bigg( \frac{\tau_{1_V}} {\mathbb{E}_{1_U}[\tau_{1_V}]} > x \bigg) = \int_x^{\infty} \mathcal{P}_{\mathrm{sup}}(y) \,dy = \left\{\begin{array}{ll}
1, &\text{ if } x \in [0,1), \\[0.2cm]
0, &\text{ if } x \in [1,\infty),
\end{array}
\right.
\qquad x \in [0,\infty),
\end{equation}
with 
\begin{equation}
\mathcal{P}_{\mathrm{sup}}(z) = \delta_1 (z), \qquad z \in [0,\infty),
\end{equation}
where $\delta_1(z)$ is the Dirac function at 1.
\end{itemize}
\end{teorema}

\noindent
Theorem~\ref{teoremapaper1} shows that there is a \emph{trichotomy} (see Fig.~\ref{fig:trichotomy}): depending on the value of $\beta$ the transition exhibits a \textit{subcritical regime}, a \textit{critical regime} and a \textit{supercritical regime}. A heuristic for the critical value is the following: the fraction of joint inactivity time of the nodes in $U$ is of order $(1/r^{\beta})^{|U|} = r^{-\beta |U|}$; since the time it takes to leave the joint inactivity state is of order $r^{-\beta}$, all nodes in $U$ become simultaneously inactive for the first time after a period of order $r^{-\beta}/r^{-\beta |U|}= r^{\beta(|U|-1)}$. Our goal is to extend Theorem~\ref{teoremapaper1} to arbitrary bipartite graphs (see Fig.~\ref{fig:ex} for examples). Note how the mean transition time depends on the actual value of the initial queue lengths at nodes in $U$: for complete bipartite graphs, the initial queue lengths are fixed to be $\gamma_U r$; for arbitrary bipartite graphs, we will see how the mean transition time depends on the way the queue lengths change while activating nodes in $V$.

\begin{figure}[htbp]
\begin{minipage}{0.25\linewidth}
\includegraphics[keepaspectratio=true,width=1\textwidth]{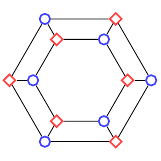}
\end{minipage}
\hfill
\begin{minipage}{0.25\linewidth}
\includegraphics[keepaspectratio=true,width=0.85\textwidth]{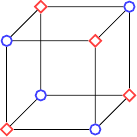}
\end{minipage}
\hfill
\begin{minipage}{0.25\linewidth}
\includegraphics[keepaspectratio=true,width=0.9\textwidth]{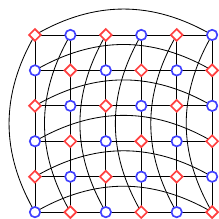}
\end{minipage}
\vspace{0.5cm}
\caption{\small Examples of bipartite graphs: cyclic ladder (left), hypercube (center), even torus (right).}
\label{fig:ex}
\end{figure}

\begin{definizione}{\bf [Key notions 2]} \label{defgeneral}
$\mbox{}$\\
{\bf (1) Neighbors of a node.} 
For a node $v \in V$, we define the set of \textit{neighbors} of $v$ as
\begin{equation} \label{neighbors}
N(v) = \{ u \in U\colon\, uv \in E \}
\end{equation}
and the degree of $v$ as
\begin{equation}
d(v) = |N(v)|.
\end{equation}
{\bf (2) Updated queue lengths.} 
Let $Q_U = \{ Q_u \}_{u \in U}$ be the sequence of queues associated with the nodes in $U$, and $Q_V = \{ Q_v \}_{v \in V}$ the sequence of queues associated with the nodes in $V$. Put $Q=(Q_U, Q_V)$, and let $Q^{k}= (Q_U^{k},Q_V^{k})$ be the pair of sequences representing the \textit{updated queue lengths} after $k$ nodes in $V$ have been activated (see Definition~\ref{updatedql} later for more details).\\
{\bf (3) Transition time and forks.} 
We denote by $\mathcal{T}_{G}^{Q}$ the \textit{transition time} of the graph $G$, i.e., \eqref{transition}, conditional on the initial queue lengths $Q=(Q_U,Q_V)$. It represents the time $\tau_{1_V}$ it takes to hit configuration $1_V$ starting from configuration $1_U$. Given a node $v \in V$, we refer to the {\it fork} of $v$ as the complete bipartite subgraph of $G$ containing only node $v$, its neighbors $N(v) \subseteq U$ and the edges between them. We talk about a $d$-fork when $d(v) = d$ with $d \in \mathbb{N}$. \\
{\bf (4) Nucleation times.}
We denote by $\mathcal{T}_v^Q$ the \textit{nucleation time} of the fork of $v$ conditional on the state of the queues $Q$. It represents the time it takes for the fork of $v$ to deactivate $N(v)$ and activate $v$ and it can be seen as the transition time of the complete bipartite subgraph of $G$ represented by the fork of $v$. Note that, for $v,w \in V$, $\mathcal{T}_v^Q $ and $\mathcal{T}_w^Q$ are dependent random variables when $N(v) \cap N(w) \neq \emptyset$.
\end{definizione}
The difference of wording between transition and nucleation is chosen in order to distinguish between the full transition of $G$ and the successive nucleations of the forks (subgraphs of $G$) of the activating nodes in $V$.


\subsection{Key idea and outline}
\label{ss:outline}

The key idea behind the present paper is to define a \textit{randomized greedy algorithm} that allows us to identify the set of paths $\mathcal{A}$ the network is mostly likely to follow while deactivating the nodes in $U$ and activating the nodes in $V$. We label the nodes in $V$ based on their first activation and we denote by $a^*$ the path that the network follows, i.e., $a^*= \{ v_1^*, \dots, v_N^* \}$ with $v_1^*$ the first node that activates and $v_N^*$ the last. Let $\mathcal{E}(a^*)$ denote the event that one of the paths in $\mathcal{A}$ occurs. We will prove that 
\begin{equation}
\lim_{r\to\infty} \mathbb{P}_{1_U}(\mathcal{E}(a^*)) =1.
\end{equation} 
In particular, we will show that if we condition on the event
\begin{equation}
A_a = \{ \text{the network follows path } a \in \mathcal{A} \},
\end{equation}
then we are able to identify how the mean transition time $\mathbb{E}_{1_U}[\mathcal{T}_{G}^{Q} \,|\, A_a]$ depends on the sequence of nucleation times of the forks of the nodes in $V$, ordered as in the path $a$ (Theorem~\ref{mostlikelypaths} below). We derive the asymptotics of the mean transition time as $r \to \infty$ (Theorem~\ref{meantransitiontime} below) and identify the law of the transition time on the scale of its mean (Theorem~\ref{law} below). To do so, we determine how the queue lengths change along the given path (Theorem~\ref{qlthm} below). Similarly as for the complete bipartite graph in Theorem~\ref{teoremapaper1}, we distinguish between three regimes for the value of $\beta$ (subcritical, critical and supercritical), in which the queues behave differently and, consequently, so does the transition time. 

\paragraph{Outline.}
The remainder of the paper is organized as follows. In Section~\ref{sec:alg} we introduce the algorithm, show that it has two important properties -- \emph{greediness} and \emph{consistency} -- and give an example of how it works. In Section~\ref{sec:theorems} we state our main theorems. In particular, we show how both the mean transition time and the law of the transition time on the scale of its mean can be determined according to the path that the algorithm chooses. In Section~\ref{sec:4} we show how the nucleation times depend on the graph structure and we analyze how the queue lengths at the nodes change along each path that the algorithm chooses. In Section~\ref{sec:algprop} we provide the proof of the two algorithm properties mentioned above. In Section~\ref{sec:thproof} we prove our main theorems. In Appendix A, we show some technical computations for the mean nucleation time in the special setting of independent forks competing for activation.


\section{Algorithm}
\label{sec:alg}

In this section we introduce a randomized algorithm that describes, step by step, how the network behaves while deactivating the nodes in $U$ and activating the nodes in $V$. The presentation is organized into a series of definitions and lemmas. In Section~\ref{subsec:proc} we define how the algorithm works iteratively. In Section~\ref{subsec:prop} we show that the algorithm is greedy and consistent (Propositions~\ref{greedinessprop}--\ref{consistencyprop} below). In Section~\ref{subsec:mot} we explain how the algorithm is used to capture the nucleation of the forks.  An example of a bipartite graph and how the algorithm acts on it are given in Section~\ref{subsec:example}.


\subsection{Definition of the algorithm}
\label{subsec:proc}

The algorithm takes as input the bipartite graph $G = ((U,V), E)$ and gives as output a sequence of triples that is needed to characterise the transition time, namely, 
\begin{equation}
\label{algoutput}
G \rightarrow (Y_k, \bar{d}_k, n_k)_{k=1}^{N}, 
\end{equation}
where $Y_k$ is a random variable with values in $\{1, \dots, N \}$ describing the index of the node in $V$ selected at step $k$, $\bar{d}_k \in \mathbb{N}$ is the degree of the selected node in the graph remaining after $k-1$ steps, and $n_k \in \mathbb{N}$ is a parameter that counts how many possibilities there are at step $k$ to choose the next node in $V$ (uniformly at random) from the remaining nodes with least degree. Sometimes we will write $v_k^*$ instead of $v_{Y_k}$ to emphasise that the network is following a specific order while activating the nodes. The integer $N = |V|$ represents the number of iterations of the algorithm.

\begin{definizione}{\bf [Algorithm]} 
\label{defYk}
Set $G= G_1= ((U_1,V_1),E_1)$. For $k=1, \dots, N$, given $G_k= ((U_k,V_k),E_k)$, find $G_{k+1}=((U_{k+1},V_{k+1}),E_{k+1})$ by iterating the following procedure until $V_{k+1}$ is empty:
\begin{itemize}
\item Start from the graph $G_k$.
\item Look at the nodes in $V_k$ and at their \emph{minimum degree} $\bar{d}_{k}$ in $G_k$. 
\item Pick a node in $V_k$ \emph{uniformly at random} from the ones with minimum degree $\bar{d}_{k}$. 
\item Denote the chosen node by $v_k^*$ and the \emph{number of choices} by $n_{k}$.
\item Eliminate the node $v_k^*$, all its neighbors in $U_k$, together with all their edges. Denote the resulting bipartite graph by $G_{k+1}$.
\end{itemize}
\end{definizione}

We next introduce the notion of admissible paths, which will be relevant to computing the transition time of the graph with the help of the above algorithm.

\begin{definizione}{\bf[Paths and admissible paths]}
Define a \textit{path} $a = (v_1, \dots, v_N)$ as a sequence of activating nodes in $V$, where each node is present exactly once. The set of paths, denoted by $\Omega$, is the set of every possible ordering of the nodes in $V$. Let $\mathcal{A}$ be the set of \textit{admissible paths} defined as the set of all the paths generated by the algorithm with positive probability. If $a^*$ is the path followed by the network, we define $A_a$ to be the event that the network follows the admissible path $a \in \mathcal{A}$ as
\begin{equation}
\label{networkfollowspatha}
A_a = \{ a^* = a \}.
\end{equation}
\end{definizione}

We are now ready to define the transition time along an admissible path.

\begin{definizione}{\bf [Transition time along an admissible path]}
Consider an admissible path  $a=(v_1,\ldots,v_N) \in \mathcal{A}$. The transition time along path $a$ is denoted by $\mathcal{T}_G^{Q} \,|\, A_a $ and it is defined as the transition time of $G$ conditional on the order of activating nodes being as in $a$ and on the initial queue lengths $Q$.
\end{definizione}

The idea of eliminating step by step the nodes in $U$ that deactivate comes from the fact that when a node in $V$ activates, it ``blocks'' all its neighbors in $U$, which with high probability will remain inactive for the rest of the time. This is due to the aggressiveness of the nodes in $V$ compared to the nodes in $U$ (recall \eqref{aggress1}--\eqref{beta'}). 

\begin{lemma}{\bf [Activation sticks]} 
\label{blockedU}
Consider a node $u \in U$ and let $N(u) \subseteq V$ be the set of neighbors of $u$. Denote by $t_u$ the first time a node $v \in N(u)$ activates. Then, with high probability $u$ remains inactive after $t_u$ for the duration of the transition, i.e., $X_u(t) = 0$ for all $t_u \leq t \leq T_G^Q$.
\end{lemma}

The above lemma will be proved in Section~\ref{subsec:prtwolem} and ensures that the transition along an admissible path can be decomposed into a succession of nucleations associated with the nodes in the path.

\begin{remark}{[\bf Good behavior]} \label{gbremark}
Recall from \cite{BdHNS18} that the queue lengths at nodes in $U$ all have a \textit{good behavior}, in the sense that the queue length at $u \in U$ stays close to its mean until one of the neighbors in $V$ activates or until time $T_U(r) =  \frac{\gamma_U}{c-\rho_U} r \, [1+o(1)]$, $r \to \infty$, the expected time it takes for the queue lengths at nodes in $U$ to hit zero. More precisely, for any $u \in U$, if we denote by $T_{u}^{\mathrm{gb}}$ the minimum between these two times, for $\delta > 0$ small enough and for all $t \in [0, T_{u}^{\mathrm{gb}}]$,
\begin{equation}
\lim_{r \to \infty} \mathbb{P}_{1_U}\bigg(\mathbb{E}_{1_U}[Q_u(t)] - \delta r \leq Q_u(t) \leq \mathbb{E}_{1_U}[Q_u(t)] + 2\delta r \bigg) = 1.
\end{equation}
Since the queue lengths at nodes in $U$ always remain of order $r$ while subcritical or critical nodes are activated, it follows that, for any $u \in U$ and for all $t \in [0, T_{u}^{\mathrm{gb}}]$,
\begin{equation} \label{gbeq}
Q_u(t) =  \mathbb{E}_{1_U}[Q_u(t)]\,[1 + o(1)], \qquad r \to \infty,
\end{equation}
which means that the queue length is always close to its mean for all times smaller than $T_{u}^{\mathrm{gb}}$.
\end{remark}

Any statement involving the transition time or the nucleation times holds for typical values of the queue lengths, i.e., for values compatible with \textit{good behavior}. With the above remark in mind, in accordance with Theorem~\ref{teoremapaper1}, we now define the mean nucleation times associated with the nodes of an admissible path.

\begin{definizione}{\bf [Nucleation times associated with an admissible path]} \label{defnucleationtime}
Suppose that the algorithm generates the admissible path $(v_1^*,\ldots,v_N^*)$. Associated with each step $k$ of the algorithm is the nucleation time $\mathcal{T}_{v_k^*}^{Q^{k-1}}$ of the fork of node $v_k^*$ (see Definition~\ref{defgeneral}), which satisfies, for any $u \in U_k$,
\begin{equation} \label{nucleationalg}
\mathbb{E}_{1_U} [\mathcal{T}_{v_k^*}^{Q^{k-1}}] = F^k \, (Q_u^{k-1})^{1 \wedge \beta(\bar{d}_k-1)}\, [1+o(1)],  \qquad r \to \infty.
\end{equation}
Here $F^k$ is a pre-factor that depends on the degree $\bar{d}_k$, which plays the role of $|U|$ in Theorem~\ref{teoremapaper1}, and on its relation with $\beta$. The term $Q_u^{k-1}$ represents the updated queue length at node $u \in U_k$ in the subgraph $G_k$ and plays the role of the initial queue lengths in Theorem~\ref{teoremapaper1}.
\end{definizione}
Note that \eqref{nucleationalg} is an asymptotic statement about sequences of random variables (see the following remark). The notation $o(1) = o_{\mathbb{P}_{1_U}}(1)$ refers to a random variable determined by the law $\mathbb{P}_{1_U}$ that goes to 0 in distribution as $r \to \infty$. Similarly, for $\alpha > 0$, the notation $o(r^{\alpha}) = o_{\mathbb{P}_{1_U}}(r^{\alpha})$ refers to a random variable determined by the law $\mathbb{P}_{1_U}$ that goes to 0 in distribution when divided by $r^{\alpha}$ as $r \to \infty$. Throughout the paper we write $o(1)$ and $o(r^{\alpha})$ to simplify the notation.

\begin{remark}{\bf[Conditioning]}
\label{conditionalrmk}
Every time we consider the transition time or the nucleation times we are conditioning on the state of the queues, hence all the expectations should be interpreted as conditional expectations. While the initial queue lengths are fixed via \eqref{initialqueues}, the updated queue lengths are random. For $v \in V$, in expressions of the form $\mathcal{T}_v^{Q^{k-1}}$ with $k =1, \dots, N$, the dependence on the random updated queue lengths is indicated by the superscript $Q^{k-1}$ (note that $Q^0$ actually represents the initial queue lengths). More precisely, when we consider step $k$ of the algorithm, or the induced subgraph $G_k$, for $k = 1, \dots, N$, we are conditioning on $Q^{k-1}$ and on the first $k-1$ activating nodes. Throughout the paper for compactness we omit the specific notation for conditional random variables and conditional expectations, but the reader is encouraged to keep it in mind for the statements that follow.

\end{remark}

Intuitively, the sum of the mean nucleation times associated with an admissible path gives the mean transition time along that path. We will see in Section~\ref{subsec:nextnucl} that the pre-factors $F^k$ actually need to be adjusted by certain weights that depend on the graph structure. 


\subsection{Properties of the algorithm} 
\label{subsec:prop}

\begin{definizione}{\bf [Maximum least degree]}\label{maxleast}
Given the sequence $(\bar{d}_k)_{k=1}^N$ generated by the algorithm, let $d^* = \max_{1 \leq k \leq N} \bar{d}_k$ be the \textit{maximum least degree} of the path associated with $(\bar{d}_k)_{k=1}^N$. 
\end{definizione}

The notions of minimum degree $\bar{d}_k$ at step $k$ and maximum least degree $d^*$ can be extended also to non-admissible paths. For a general path, the degree $\bar{d}_k$ at step $k$ is the minimum degree of the remaining nodes in $V$ in the induced subgraph of $G$ obtained by removing the nodes activated in the first $k-1$ steps and their neighbors. We will show that the set of admissible paths $\mathcal{A}$ is the set of the {\it most likely} paths the network follows. The following lemma and two propositions will be proved in Section~\ref{subsec:grconst}.

\begin{lemma}{\bf [Comparing maximum least degrees of different paths]} \label{lemmaab}
Consider two different paths $a,b$ such that $a \in \mathcal{A}$ is admissible. For $k = 1,\ldots,N$, denote by $\bar{d}_{k,a}$ and $\bar{d}_{k,b}$ the minimum degrees at step $k$ in paths $a$ and $b$. Let $d_a^* = \max_{1 \leq k \leq N} \bar{d}_{k,a}$ and $d_b^* = \max_{1 \leq k \leq N} \bar{d}_{k,b}$. Then $d_a^* \leq d_b^*$.
\end{lemma}
\noindent
In other words, given any path $b$, its maximum least degree cannot be smaller than the maximum least degree of an admissible path $a$. We will see how the maximum least degree $d^*$ determines the order of the mean transition time. Depending on how $\beta$ is related to $d^*$, we distinguish three different regimes: 
\begin{equation}
\begin{array}{lll}
&\text{\emph{subcritical:}} &\beta \in (0, \frac{1}{d^*-1}),\\
&\text{\emph{critical:}} &\beta = \frac{1}{d^*-1},\\
&\text{\emph{supercritical:}} &\beta \in (\frac{1}{d^*-1}, \infty).
\end{array}
\end{equation}

Note that $d^* =0$ means that there are no edges in the graph, while $d^*=1$ means that each node in $V$ has at most one neighbor in $U$ at the moment of its activation, which implies that the transition occurs in time $O(1)$. The most interesting scenarios to investigate occur then when $d^* \geq 2$, which will be an implicit assumption throughout the paper.

The algorithm is \emph{greedy}, in the sense that it always chooses a node that adds the least to the order of the total transition time along the path, simply because this node is likely to be the first to activate.

\begin{proposizione}{\bf [Greediness]} \label{greedinessprop}
The order of the mean transition time along any admissible path is the smallest possible.
\end{proposizione}

\noindent
The algorithm is \emph{consistent}, in the sense that $d^*$ is unique. Different admissible paths lead to the same order of the mean transition time.

\begin{proposizione}{\bf [Consistency]} \label{consistencyprop}
All the admissible paths lead to the same order of the mean transition time.
\end{proposizione}


\subsection{Structure of the algorithm}
\label{subsec:mot}

It is intuitive that a node in $V$ activates because it is the one whose complete bipartite fork has the fastest nucleation. Note that this depends on the activation and deactivation Poisson clocks, and on the queue length processes. We will see in Theorem~\ref{mostlikelypaths} below that, with high probability, the network follows an admissible path. Hence a node in $V$ activates because it is the one whose complete bipartite fork has the fastest nucleation among the nodes with minimum degree.

\begin{definizione}{\bf [Next nucleation time]} \label{mintime}
On the event that the network follows one of the admissible paths, given that $k-1$ nodes in $V$ have already been activated, define the time the network subsequently takes to activate the $k$-th node in $V$ by
\begin{equation}
\label{nntminimum}
\bar{\tau}_k = \mathrm{min}_{v \in V_k} \mathcal{T}_{v}^{Q^{k-1}},
\end{equation}
where $V_k$ is the set of inactive nodes in $V$ produced by the algorithm after $k-1$ iterations.
\end{definizione}

Note that if we condition on the event $A_a$ that the network follows the admissible path $a = (v_1, \dots, v_N) \in \mathcal{A}$, then the $k$-th activating node $a_k$ is the node that realizes the minimum in \eqref{nntminimum}, and its nucleation time is $\mathcal{T}_{a_k}^{Q^{k-1}} = \bar{\tau}_k$. By keeping track of which nodes have been picked, in Section~\ref{subsec:uql} we will compute the updated queue lengths after each activation in $V$, of which we are now able to give a precise definition (recall Definition~\ref{defgeneral}(2)).

\begin{definizione}{\bf [Updated queue lengths]} \label{updatedql}
On the event that the network follows one of the admissible paths, for $k = 1, \dots, N$, define the \textit{updated queue lengths} $Q^{k-1}$ at step $k$ after $k-1$ nodes have been activated by
\begin{equation}
Q^{k-1} = \left(Q_U^{k-1}, Q_V^{k-1}\right) = \left(\{Q_u^{k-1}\}_{u \in U}, \{Q_v^{k-1}\}_{v \in V}\right),
\end{equation}
where 
$$\left\{Q_u^{k-1}\right\}_{u \in U} = \left\{Q_u\left(\sum_{l=1}^{k-1} \bar{\tau}_l \right) \right\}_{u \in U}$$ 
and 
$$\left\{Q_v^{k-1}\right\}_{v \in V} = \left\{Q_v\left(\sum_{l=1}^{k-1} \bar{\tau}_l \right) \right\}_{v \in V}$$ 
are vectors that represent the updated queue lengths at nodes in $U$ and $V$, respectively.
\end{definizione}

When a node in $V$ activates, its fork can be of three different types depending on how its degree is related to $\beta$. 

\begin{definizione}{\bf [Subcritical, critical and supercritical nodes]}
\label{def:regimes}
Given that $k-1$ nodes in $V$ have already been activated, consider the $k$-th activating node and its fork of degree $\bar{d}_k$. 
\begin{itemize}
\item If $\beta(\bar{d}_k-1) <1$, then the node (or its fork) is subcritical.
\item If $\beta(\bar{d}_k-1) =1$, then the node (or its fork) is critical.
\item If $\beta(\bar{d}_k-1) >1$, then the node (or its fork) is supercritical.
\end{itemize}
\end{definizione}

In the subcritical and the critical regime, the nucleation time of a node from Definition~\ref{mintime} is with high probability a minimum over the nodes with least degree in $V_k$. Indeed, nodes with least degree activate first with high probability. The following lemma will be proved in Section~\ref{subsec:prtwolem}.

\begin{lemma}{\bf [Activation selects low degree]}
\label{lemmavw}
For $k = 1, \dots, N$, given $v,w \in V_k$ such that $d_k(w) > d_k(v)= \bar{d}_k$ and $\beta(\bar{d}_k-1) \leq 1$, the conditional probability of $w$ activating before $v$ satisfies
\begin{equation}
\lim_{r \to \infty} \mathbb{P}_{1_U} \big(\mathcal{T}_w^{Q^{k-1}} < \mathcal{T}_v^{Q^{k-1}}\big) = 0.
\end{equation}
\end{lemma}

In the supercritical regime the situation is more delicate. If at step $k$ the least degree fork has degree $\bar{d}_k$ such that $\beta(\bar{d}_k-1) >1$, then the mean nucleation time of the next activating fork is the same for all the remaining forks in the graph. The network does not distinguish between the nodes according to their degree anymore, since all possibilities contribute equally to the total mean transition time. Indeed, we know from Theorem~\ref{teoremapaper1} that the mean nucleation time is given by the expected time it takes for the queues in $U$ to hit zero. Hence, after the nucleation of the first supercritical fork, all the queues in $U$ are of order $o(r)$ and the transition occurs very fast (see Section~\ref{subsec:uql} for more details). 

In Section~\ref{sec:theorems} we will see how the transition time can be determined given the set of admissible paths. Moreover, we will identify the mean transition time along each path and its law on the scale of its mean. Given a path, we know in which order the nodes activate. In Section~\ref{sec:thproof} we will see how we can identify the nucleation time of a node given in Definition~\ref{mintime} with the nucleation time of the complete bipartite fork of the activating node, as written in \eqref{nucleationalg}. The sum of all the nucleation times gives us the transition time of the graph. Not all the terms in the sum contribute significantly in the limit as $r \to \infty$. We will need to identify which are the leading order terms. The answer depends on the sequence of degrees $(\bar{d}_{k})_{k=1}^N$ generated by the algorithm and on how the queue lengths change along the path.


\subsection{Example}
\label{subsec:example}

Consider the bipartite graph $G=((U,V),E)$ with $|U| = 6$ and $|V| = 4$ in Fig.~\ref{fig:switch1}. This graph serves as a simple example of how the algorithm works.

\vspace{0.2cm}
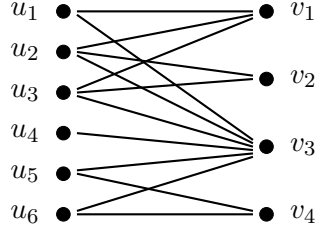
\begin{figure}[htbp]
\begin{center}
\begin{tikzpicture}[scale = 0.9]
\draw[fill] (0,0) circle (0.1);
\draw[fill] (0,1.8) circle (0.1);
\draw[fill] (0,2.4) circle (0.1);
\draw[fill] (0,3) circle (0.1);
\draw[fill] (0,0.6) circle (0.1);
\draw[fill] (0,1.2) circle (0.1);
\draw[fill] (3,0) circle (0.1);
\draw[fill] (3,1) circle (0.1);
\draw[fill] (3,2) circle (0.1);
\draw[fill] (3,3) circle (0.1);
\node [left] at (-0.2,0) {$u_6$};
\node [left] at (-0.2,0.6) {$u_5$};
\node [left] at (-0.2,1.2) {$u_4$};
\node [left] at (-0.2,1.8) {$u_3$};
\node [left] at (-0.2,2.4) {$u_2$};
\node [left] at (-0.2,3) {$u_1$};
\node [right] at (3.2,0) {$v_4$};
\node [right] at (3.2,1) {$v_3$};
\node [right] at (3.2,2) {$v_2$};
\node [right] at (3.2,3) {$v_1$};
\draw[thick] (0.2,3) -- (2.8,3);
\draw[thick] (0.2,2.95) -- (2.8,1.1);
\draw[thick] (0.2,2.45) -- (2.8,2.95);
\draw[thick] (0.2,2.4) -- (2.8,2.05);
\draw[thick] (0.2,2.35) -- (2.8,1.05);
\draw[thick] (0.2,1.85) -- (2.8,2.9);
\draw[thick] (0.2,1.8) -- (2.8,2);
\draw[thick] (0.2,1.75) -- (2.8,1);
\draw[thick] (0.2,1.2) -- (2.8,0.95);
\draw[thick] (0.2,0.65) -- (2.8,0.9);
\draw[thick] (0.2,0.6) -- (2.8,0.05);
\draw[thick] (0.2,0.05) -- (2.8,0.85);
\draw[thick] (0.2,0) -- (2.8,0);
\end{tikzpicture}
\end{center}
\caption{{\small The initial bipartite graph $G= G_1= ((U_1,V_1), E_1)$.}}
\label{fig:switch1}
\end{figure}

\medskip\noindent
\underline{{\bf $k=1$}}. 
We start with $G= G_1= ((U_1,V_1),E_1)$. There are two nodes $v_2, v_4$ with minimum degree $\bar{d}_1 = 2$, so $n_1 = 2$. Pick uniformly at random one of them (with probability $\frac{1}{n_1} = \frac{1}{2}$), say $Y_1 = 2$. Eliminate node $v_2$, all its neighbors $u_2, u_3$, and all their edges $u_2v_1, u_2v_2, u_2v_3, u_3v_1, u_3v_2, u_3v_3$. Denote the new bipartite graph by $G_2=((U_2,V_2),E_2)$. The nucleation time associated with this node satisfies, for any $u \in U_1$,
\begin{equation} \label{1}
\mathbb{E}_{1_U} [\mathcal{T}_{v_{Y_1}}^{Q^0}] = \mathbb{E}_{1_U} [\mathcal{T}_{v_2}^{Q^0}] = F^1\, (Q_u^0)^{1 \wedge \beta}\, [1+o(1)], \qquad r \to \infty.
\end{equation}

\medskip\noindent
\underline{{\bf $k=2$}}.
Node $v_1$ has the minimum degree $\bar{d}_2 = 1$, so $Y_2 = 1$. Eliminate node $v_1$, all its neighbors, and all their edges. Denote the new bipartite graph by $G_3=((U_3,V_3),E_3)$. The nucleation time associated with this node satisfies, for any $u \in U_2$,
\begin{equation} \label{2}
\mathbb{E}_{1_U} [\mathcal{T}_{v_{Y_2}}^{Q^1}] = \mathbb{E}_{1_U} [\mathcal{T}_{v_1}^{Q^1}] = F^2\, (Q_u^1)^0 \, [1+o(1)] = O(1), \qquad r \to \infty.
\end{equation}

\medskip\noindent
\underline{{\bf $k=3$}}.
Node $v_4$ has the minimum degree $\bar{d}_3 = 2$, so $Y_3 = 4$. Eliminate node $v_4$, all its neighbors, and all their edges. Denote the new bipartite graph by $G_4=((U_4,V_4),E_4)$. The nucleation time associated with this node satisfies, for any $u \in U_3$,
\begin{equation} \label{3}
\mathbb{E}_{1_U} [\mathcal{T}_{v_{Y_3}}^{Q^2}] = \mathbb{E}_{1_U} [\mathcal{T}_{v_4}^{Q^2}] = F^3\, (Q_u^2)^{1 \wedge \beta}\, [1+o(1)], \qquad r \to \infty.
\end{equation}

\medskip\noindent
\underline{{\bf $k=4$}}.
Node $v_3$ is the only node left, with degree $\bar{d}_4 = 1$, so $Y_4 = 3$. Eliminate node $v_3$, all its neighbors, and all their edges, after which the empty graph is left. The nucleation time associated with this node satisfies, for any $u \in U_4$,
\begin{equation} \label{4}
\mathbb{E}_{1_U} [\mathcal{T}_{v_{Y_4}}^{Q^3}] = \mathbb{E}_{1_U} [\mathcal{T}_{v_3}^{Q^3}] = F^4\, (Q_u^3)^0 \, [1+o(1)] = O(1), \qquad r \to \infty.
\end{equation}
The above scenario forms a path that is described by nodes in $V$ activating in the order $v_2, v_1, v_4, v_3$ (see Fig.~\ref{fig:scenario}). 

\begin{figure}[htbp] \label{switch2} 
\begin{center}
\begin{tikzpicture}[scale = 0.9]
\draw[fill] (0,0) circle (0.1);
\draw[fill] (0,0.6) circle (0.1);
\draw[fill] (0,1.2) circle (0.1);
\node at (0,1.8) {$\times$};
\node at (0,2.4) {$ \times $};
\draw[fill] (0,3) circle (0.1);
\draw[fill] (3,0) circle (0.1);
\draw[fill] (3,1) circle (0.1);
\node at (3,2) {$ \times $};
\draw[fill] (3,3) circle (0.1);
\node [left] at (-0.2,0) {$u_6$};
\node [left] at (-0.2,0.6) {$u_5$};
\node [left] at (-0.2,1.2) {$u_4$};
\node [left] at (-0.2,3) {$u_1$};
\node [right] at (3.2,0) {$v_4$};
\node [right] at (3.2,1) {$v_3$};
\node [right] at (3.2,3) {$v_1$};
\draw[thick] (0.2,3) -- (2.8,3);
\draw[thick] (0.2,2.95) -- (2.8,1.1);
\draw[thick] (0.2,1.2) -- (2.8,0.95);
\draw[thick] (0.2,0.65) -- (2.8,0.9);
\draw[thick] (0.2,0.6) -- (2.8,0.05);
\draw[thick] (0.2,0.05) -- (2.8,0.85);
\draw[thick] (0.2,0) -- (2.8,0);
\draw[fill] (6,0) circle (0.1);
\draw[fill] (6,0.6) circle (0.1);
\draw[fill] (6,1.2) circle (0.1);
\node at (6,1.8) {$\times$};
\node at (6,2.4) {$ \times $};
\node at (6,3) {$\times$};
\draw[fill] (9,0) circle (0.1);
\draw[fill] (9,1) circle (0.1);
\node at (9,2) {$ \times $};
\node at (9,3) {$\times$};
\node [left] at (5.8,0) {$u_6$};
\node [left] at (5.8,0.6) {$u_5$};
\node [left] at (5.8,1.2) {$u_4$};
\node [right] at (9.2,0) {$v_4$};
\node [right] at (9.2,1) {$v_3$};
\draw[thick] (6.2,1.2) -- (8.8,0.95);
\draw[thick] (6.2,0.65) -- (8.8,0.9);
\draw[thick] (6.2,0.6) -- (8.8,0.05);
\draw[thick] (6.2,0.05) -- (8.8,0.85);
\draw[thick] (6.2,0) -- (8.8,0);
\node at (12,0) {$ \times $};
\node at (12,0.6) {$ \times $};
\draw[fill] (12,1.2) circle (0.1);
\node at (12,1.8) {$\times$};
\node at (12,2.4) {$ \times $};
\node at (12,3) {$\times$};
\node at (15,0) {$ \times $};
\draw[fill] (15,1) circle (0.1);
\node at (15,2) {$ \times $};
\node at (15,3) {$\times$};
\node [left] at (11.8,1.2) {$u_4$};
\node [right] at (15.2,1) {$v_3$};
\draw[thick] (12.2,1.2) -- (14.8,0.95);
\end{tikzpicture}
\end{center}
\caption{{\small The sequence of bipartite graphs $G_2= ((U_2,V_2),E_2)$, $G_3= ((U_3,V_3),E_3)$, $G_4= ((U_4,V_4),E_4)$ generated by the algorithm.}}
\label{fig:scenario}
\end{figure}
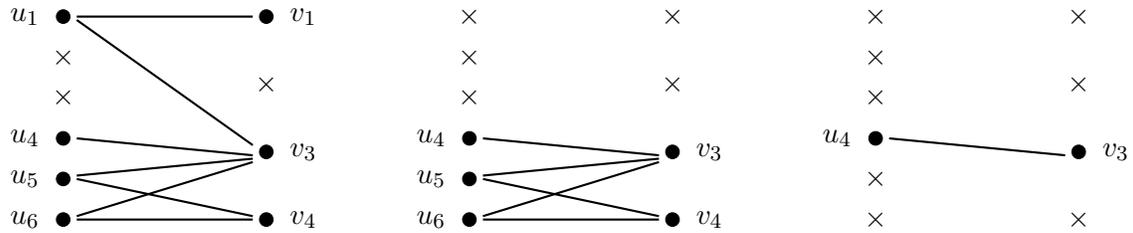

Note that the algorithm may pick node $v_4$ at the first step by setting $Y_1 = 4$, since the choice of the node with minimum degree is made uniformly at random. If so, then the algorithm follows a different path. At the first step, $Y_1 =4$ and $\mathbb{E}_{1_U} [\mathcal{T}_{v_4}^{Q^0}] =  F^1 \,(Q_u^0)^{1 \wedge \beta}\, [1+o(1)]$ for any $u \in U_1$. At the second step, $Y_2 = 2$ and $\mathbb{E}_{1_U} [\mathcal{T}_{v_2}^{Q^1}] = F^2 \,(Q_u^1)^{1 \wedge \beta}\, [1+o(1)]$ for any $u \in U_2$. At the third step, $Y_3 = 1$ and $\mathbb{E}_{1_U} [\mathcal{T}_{v_1}^{Q^2}] = O(1)$. At the fourth step, $Y_4 = 3 $ and $\mathbb{E}_{1_U} [\mathcal{T}_{v_3}^{Q^3}] = O(1)$. This choice leads to a different path, where the nodes in $V$ activate in the order $v_4, v_2, v_1, v_3$.

Each possible scenario is identified with a path in the algorithm (admissible path), described by the nodes in $V$ according to the order of their first activation. The total mean transition time along a path can be thought of as a sum of the mean nucleation times associated with each activating node in the path (see Theorem~\ref{mostlikelypaths}). We will prove in Section~\ref{subsec:grconst} that all the admissible paths lead to the \emph{same} order of the mean transition time.


\section{Transition time: main theorems}
\label{sec:theorems}

In this section we present our main theorems regarding the transition time. In Section~\ref{subsec:mostlike} we show that $\mathcal{E}$, the event that the network follows the algorithm, has high probability (Theorem~\ref{mostlikelypaths}(i) below). We analyze the contributions along a given admissible path, noting that not all the nucleation times are significant for the total mean transition time (Theorem~\ref{mostlikelypaths}(ii) below). In Section~\ref{subsec:meantr} we compute the asymptotics of the mean transition time, including the pre-factor, focusing on the significant terms only (Theorem~\ref{meantransitiontime} below). In Section~\ref{subsec:lawtr} we identify the law of the transition time divided by its mean, which turns out to be a convolution of the laws found for the complete bipartite graph in Theorem~\ref{teoremapaper1} (Theorem~\ref{law} below). There is again a trichotomy, depending on the value of $\beta$. Proofs will be given in Section~\ref{sec:thproof}.


\subsection{Most likely paths}
\label{subsec:mostlike}

Recall that $\Omega$ is the set of all possible paths and $\mathcal{A} \subseteq \Omega $ is the set of admissible paths. Denote by $\mathcal{A}_{sc}$ the subset of admissible paths \emph{truncated at the first supercritical node} (if there is any). Recall that, according to Definition~\ref{def:regimes}, a supercritical node is a node that is activated through a supercritical fork. If $a= (v_1, \dots, v_N) $ is an element of $\mathcal{A}$, then $a_{sc} = (v_1, \dots, v_{sc})$ is an element of $\mathcal{A}_{sc}$, where $v_{sc}$ denotes the last node of each truncated ordering. We allow this node to be any of the remaining supercritical nodes. 

\begin{definizione}{\bf [The network follows the algorithm]} \label{eventE}
Denote by $a^* = (v^*_1, \dots, v^*_N)$ the path followed by the network and define
\begin{equation}
\mathcal{E}(a^*) =
\Big\{ \exists \, a_{sc} = (v_1, \dots, v_{sc}) \in \mathcal{A}_{sc}\colon\, v_1 = v_1^*, \dots, v_{sc} = v^*_{sc} \Big\} 
\end{equation} 
as the event that the network follows any of the admissible paths up to the first supercritical node (if there is any).
\end{definizione}

Our first main theorem consists of three statements and shows how the algorithm helps us find the mean transition time of the network. The first statement holds for all three regimes. The second and third statements hold for the subcritical and the critical regime only (for which the network follows the algorithm until the last activating node). The idea is that the mean transition time of the network can be seen as a weighted sum of the mean nucleation times associated with each activation and of negligible terms representing the time it takes after each activation to bring the network back in the configuration with all the nodes in $U$ active. For the second statement, we denote by $\mathcal{A}_k$ the set of (partial) admissible paths until step $k$. For the third statement, recall that $A_a$ is the event that the network follows the admissible path $a \in \mathcal{A}$, as defined in \eqref{networkfollowspatha}. Note that in both the statements we take the expectation $\mathbb{E}_Q$ on the right-hand side: this averages over the random values $Q^1, \dots, Q^{N-1}$ of the updated queue lengths on which we condition via the nucleation times appearing in each term of the sums. For the supercritical regime we do not need any statement, because the mean transition time is known from \cite{BdHNS18} to be the expected time it takes for the queue lengths to hit zero.

\begin{teorema}{\bf [Most likely paths]} \label{mostlikelypaths}
Consider the bipartite graph $G = ((U,V),E)$ with initial queue lengths $Q^0=(Q_U^0,Q_V^0)$ as in \eqref{initialqueues}.
\begin{itemize}
\item[{\rm (i)}] 
With high probability the network follows the algorithm, i.e.,
\begin{equation}
\lim_{r \to \infty} \mathbb{P}_{1_U}(\mathcal{E}(a^*)) = 1.
\end{equation}
\end{itemize}
Consider $\beta \in (0, \frac{1}{d^*-1}]$: subcritical or critical regime. 
\begin{itemize}
\item[{\rm (ii)}]
With high probability the transition time of $G$ given the initial queue lengths $Q^0$ satisfies
\begin{equation}
\label{sumnucl}
\mathbb{E}_{1_U}[\mathcal{T}_G^{Q^0} \mathbbm{1}_{\mathcal{E}(a^*)}] = \mathbb{E}_Q \left[ \sum_{k=1}^N 
\sum_{\substack{i_1, \dots, i_k\colon \\ (v_{i_1},\dots, v_{i_k}) \in \mathcal{A}_k}}
\bigg(\prod_{l = 1}^k \frac{1}{n_l}\bigg) \, f_k \, \mathbb{E}_{1_U}[\mathcal{T}_{v_{i_k}}^{Q^{k-1}}
\mathbbm{1}_{\mathcal{E}(a^*)}] \right] [1+o(1)], \qquad r \to \infty,
\end{equation}
where $n_k \in \mathbb{N}$ is the number of possible nodes that the algorithm can pick at step $k$, while the factor $f_k \in (0,1)$ (to be identified in Theorem~\ref{meantransitiontime}) comes from the fact that the node activating at step $k$ is the one that activates first among the $n_k$ nodes with the same least degree. Both $n_k$ and $f_k$ depend on the sequence of nodes that have been activated before step $k$.
\item[{\rm (iii)}] 
With high probability the transition time of $G$ along an admissible path $a \in \mathcal{A}$ given the initial queue lengths $Q^0$ satisfies
\begin{equation}
\mathbb{E}_{1_U}[\mathcal{T}_{G}^{Q^0} \,|\, A_a ] = \mathbb{E}_Q \left[\sum_{k=1}^{N} f_k\,  \mathbb{E}_{1_U}[\mathcal{T}_{v_k}^{Q^{k-1}}]\right] [1+o(1)], \qquad r \to \infty.
\end{equation}
\end{itemize}
\end{teorema}

Theorem~\ref{mostlikelypaths} will be proved in Section~\ref{subsec:mlp}. Note that the mean transition time of the graph $G$ given the initial queue lengths $Q^0$ can be split as
\begin{equation}
\label{eqsplittransition}
\mathbb{E}_{1_U}[\mathcal{T}_{G}^{Q^0}] = \mathbb{E}_{1_U}[\mathcal{T}_{G}^{Q^0} \mathbbm{1}_{\mathcal{E}(a^*)}]  
+ \mathbb{E}_{1_U}[\mathcal{T}_{G}^{Q^0}\mathbbm{1}_{\mathcal{E}(a^*)^C}].
\end{equation}
The second term on the right-hand side represents the mean transition time when the network does \emph{not} follow the algorithm, and equals
\begin{equation}
\mathbb{E}_{1_U}[\mathcal{T}_{G}^{Q^0}\mathbbm{1}_{\mathcal{E}(a^*)^C}] 
=\mathbb{E}_{1_U}[\mathcal{T}_{G}^{Q^0} \,|\, \mathcal{E}(a^*)^C] \, \mathbb{P}_{1_U}(\mathcal{E}(a^*)^C).
\end{equation}
Even though we know from Theorem~\ref{mostlikelypaths}(i) that $\mathbb{P}_{1_U}(\mathcal{E}(a^*)^C)$ tends to zero as $r \to \infty$, a priori this term may still affect the total mean transition time, since the conditional expectation may be substantial. In what follows we focus on the first term on the right-hand side of \eqref{eqsplittransition}, since this captures the \emph{typical} behavior of the network. \textit{All the results in the present paper are conditional on the high probability event $\mathcal{E}(a^*)$ that the network follows the algorithm. We will omit the conditional notation to facilitate the reading. The reader should always keep this in mind while going through the statements and the proofs.}

We will see in Theorem~\ref{meantransitiontime} below that, in the supercritical regime, the mean transition time is the expected time it takes for the queues in $U$ to hit zero, independently of which path the network took before activating the first supercritical node. Theorem~\ref{mostlikelypaths}(ii)-(iii) give us a way, in the subcritical regime and the critical regime, to split the total mean transition time into a sum of mean nucleation times of successive forks, by taking into account all the admissible paths, each with its own probability. Namely,
\begin{equation}
\begin{split}
\mathbb{E}_{1_U}[\mathcal{T}_G^{Q^0}\mathbbm{1}_{\mathcal{E}(a^*)}] 
& =  \mathbb{E}_{1_U} \bigg[\sum_{a \in \mathcal{A}}\mathcal{T}_{G}^{Q^0} \mathbbm{1}_{A_a}  \,\big|\, \mathcal{E}(a^*)\bigg] \, \mathbb{P}_{1_U}(\mathcal{E}(a^*))\\
& = \sum_{a \in \mathcal{A}} \mathbb{E}_{1_U}[\mathcal{T}_{G}^{Q^0}  \,|\,  A_a ]\, \mathbb{P}_{1_U}(A_a \,|\, \mathcal{E}(a^*))  \, \mathbb{P}_{1_U}(\mathcal{E}(a^*)).
\end{split}
\end{equation}
The above expression allows us to compute the mean transition time along a single admissible path. Note that, the probability that the network follows any path $a \in \mathcal{A}$ conditional on the event that it follows the algorithm is given by the probability that the algorithm generates path $a$, i.e., 
\begin{equation}
\mathbb{P}_{1_U}(A_a \,|\, \mathcal{E}(a^*)) = \prod_{k=1}^{N} \frac{1}{n_k}.
\end{equation}
Recall that, by Proposition~\ref{consistencyprop}, that the order of the mean transition time does not depend on which path is followed.


\subsection{Mean of the transition time}
\label{subsec:meantr}

Consider an admissible path $a \in \mathcal{A}$ and the event $A_a$ that the network follows this path. Recall that $d^* = \max_{1 \leq k \leq N} \bar{d}_k$ is the maximum degree among the sequence of minimum degrees $(\bar{d}_k)_{k=1}^N$. Let $v_k^*$ be the $k$-th activating node in path $a$. According to Definition~\ref{defnucleationtime},  for any $u \in U_k$ the mean nucleation time $\mathbb{E}_{1_U}[\mathcal{T}_{v_k^*}^{Q^{k-1}}]$ is given by
\begin{equation} \label{nt&q}
\mathbb{E}_{1_U}[\mathcal{T}_{v_k^*}^{Q^{k-1}}] =
\left\{\begin{array}{ll} 
F^{k}_{\textrm{sub}} \,(Q_u^{k-1})^{\beta(\bar{d}_k-1)}\,[1+o(1)], &\text{ if } \beta \in \big(0, \frac{1}{\bar{d}_k-1}\big),\\[0.2cm]
F^{k}_{\textrm{cr}} \,Q_u^{k-1}\,[1+o(1)], &\text{ if } \beta = \frac{1}{\bar{d}_k-1},\\[0.2cm]
F^{k}_{\textrm{sup}} \,Q_u^{k-1}\,[1+o(1)], &\text{ if } \beta \in \big(\frac{1}{\bar{d}_k-1}, \infty\big),
\end{array}
\right. \qquad r \to \infty,
\end{equation}
where $F^{k}_{\textrm{sub}}, F^{k}_{\textrm{cr}}, F^{k}_{\textrm{sup}}$ are constants depending on $\bar{d}_k, B, c, \rho_U$. Namely,
\begin{equation}
F^{k}_{\textrm{sub}} =\frac{1}{\bar{d}_k B^{-(\bar{d}_k -1)}},
\qquad F^{k}_{\textrm{cr}} = \frac{1}{\bar{d}_k B^{-(\bar{d}_k -1)} + c-\rho_U},
\qquad F^{k}_{\textrm{sup}} = \frac{1}{ c-\rho_U}.
\end{equation} 
Note that $F^{k}_{\textrm{sub}}$ really depends on $k$, while $F^{k}_{\textrm{cr}} = \frac{1}{ (1/\beta +1) B^{-(1/\beta)} + c-\rho_U}$ is the same for every critical node, and $F^{k}_{\textrm{sup}} = F_{\textrm{sup}}$ is independent of $k$. Moreover, note that the first mean nucleation time depends on the initial queue lengths $Q_U^0$ at the nodes in $U$, but in general the mean nucleation time associated with a fork depends on the queue lengths at the nodes in $U$ at the moment the fork starts the nucleation.

Our second main theorem identifies the mean transition time along a given path. 

\begin{teorema}{\bf [Mean transition time]} 
\label{meantransitiontime}
Consider the bipartite graph $G = ((U,V),E)$ with initial queue lengths $Q^0=(Q_U^0,Q_V^0)$ as in \eqref{initialqueues}. The transition time of the graph $G$ given the initial queue lengths $Q^0$ satisfies the following.
\begin{itemize}
\item[{\rm (I)}] $\beta \in (0, \frac{1}{d^*-1})$: subcritical regime.
\begin{equation}
\mathbb{E}_{1_U}[\mathcal{T}_G^{Q^0} \,|\, A_a] = \sum_{\substack{1 \leq k \leq N \\ k: \,\bar{d}_k = d^*}} f_k\, \frac{ \gamma_U^{\beta(d^*-1)}}{d^*B^{-(d^*-1)}}\, r^{\beta(d^*-1)}\,[1+o(1)], \qquad r \to \infty,
\end{equation}
with
\begin{equation}
\label{fkdef1}
f_k = \frac{1}{n_k}.
\end{equation}
\item[{\rm (II)}] $\beta = \frac{1}{d^*-1}$: critical regime. Then
\begin{equation}
\mathbb{E}_{1_U}[\mathcal{T}_G^{Q^0} \,|\, A_a] = \sum_{\substack{1 \leq k \leq N \\ k: \,\bar{d}_k = d^*}} 
f_k \, \frac{\gamma_U^{(k)}}{d^*B^{-(d^*-1)}+c-\rho_U}\, r\, [1+o(1)], \qquad r \to \infty,
\end{equation}
with
\begin{equation}
\label{fkdef2}
f_k = \frac{ \bar{d}_k B^{-(\bar{d}_k -1)} + c-\rho_U}{n_k \bar{d}_k B^{-(\bar{d}_k -1)} + c-\rho_U}
\end{equation}
and
\begin{equation} \label{defgammaupdate}
\gamma_U^{(k)} = \gamma_U - (c-\rho_U) \sum_{\substack{1 \leq i \leq k \\ i: \, \bar{d}_i = d^*}} f'_i,
\end{equation}
where for a critical node $v_i$ the coefficient $f'_i$ is defined in a recursive way as 
\begin{equation}
f_i' = \frac{1}{n_{i} \bar{d}_{i} B^{-(\bar{d}_{i}-1)} + c-\rho_U} \, \Bigg(\gamma_U - (c-\rho_U) \sum_{\substack{1 \leq j \leq i-1 \\ j:\, \bar{d}_j = d^*}} f_j'\Bigg) > 0,
\end{equation}
\begin{equation} 
f'_1 = 
\left\{\begin{array}{ll} 
\frac{1}{n_1} \,\frac{1}{\bar{d}_1 B^{-(\bar{d}_1 -1)}} \,\gamma_U, & \text{ if } \bar{d}_1 < d^*, \\[0.2cm]
\frac{1}{n_1\bar{d}_1 B^{-(\bar{d}_1 -1)} + c-\rho_U} \, \gamma_U, & \text{ if } \bar{d_1} = d^*.
\end{array} \right.
\end{equation}
\item[{\rm (III)}] $\beta \in (\frac{1}{d^*-1}, \infty)$: supercritical regime.
\begin{equation}
\mathbb{E}_{1_U}[\mathcal{T}_G^{Q^0}] = \frac{\gamma_U}{c-\rho_U} r\, [1+o(1)], \qquad r \to \infty.
\end{equation}
\end{itemize}
\end{teorema}

Theorem~\ref{meantransitiontime} will be proved in Section~\ref{subsec:tr}. Both in the subcritical and the supercritical regime, Theorem~\ref{meantransitiontime} provides \emph{explicit formulas} for the mean transition time in terms of the parameters $c,\gamma_U, \rho_U$ and $B,\beta$ in our model (recall Section~\ref{ss:model}) and the sequence of numbers $(\bar{d}_k,n_k)_{k=1}^N$ that are produced by the algorithm (recall \eqref{algoutput}), with $d^* = \max_{1 \leq k \leq N} \bar{d}_k$.  In the critical regime, however, the formula is more delicate, since the pre-factor depends on how long the critical nucleations take. Indeed, $\gamma_U^{(k)}$ in \eqref{defgammaupdate} represents the mean updated queue length at nodes in $U$ at step $k$ (see Section~\ref{subsec:uql} for more details). Note that the mean transition time in the subcritical and the critical regime depends on the path, while in the supercritical regime it does not.


\subsection{Law of the transition time}
\label{subsec:lawtr}

Theorem~\ref{mostlikelypaths} shows how the mean transition time along an admissible path is a sum of terms related to the successive mean nucleation times of complete bipartite subgraphs of $G$. Theorem~\ref{meantransitiontime} tells us that, depending on the value of $\beta$, this sum reduces to a smaller sum of only a few significant terms. It also tells us how to compute the pre-factors of these terms. 

\begin{definizione}{\bf [Multiplicity of $d^*$]}
In the subcritical regime, consider an admissible path $a \in \mathcal{A}$ and its associated degree sequence $(\bar{d}_k)_{k=1}^N$. Write $m_{\mathrm{sub}}^a$ to denote the multiplicity of $d^*$ in the path $a$, i.e.,
\begin{equation} 
\label{defk1}
m_{\mathrm{sub}}^a = |\{k\colon\,\bar{d}_k = d^*\}|.
\end{equation}
\end{definizione}

Our third main theorem identifies the law of $\mathcal{T}_G^{Q^0 }/\mathbb{E}_{1_U}[\mathcal{T}_G^{Q^0}]$. Recall the laws $\mathcal{P}_{\mathrm{sub}},\mathcal{P}_{\mathrm{cr}},\mathcal{P}_{\mathrm{sup}}$ introduced in \eqref{teoremapaper1}. Write $\circledast$ to denote convolution.

\begin{teorema}{\bf [Law of the transition time]} \label{law}
Consider the bipartite graph $G = ((U,V),E)$ with initial queue lengths $Q^0=(Q_U^0,Q_V^0)$ as in \eqref{initialqueues}.
The transition time of the graph $G$ given the initial queue lengths $Q^0$ satisfies the following.
\begin{itemize}

\item[{\rm (I)}] $\beta \in (0, \frac{1}{d^*-1})$: subcritical regime. With $f_k$ as in \eqref{fkdef1} and $m_{\mathrm{sub}}^a$ as in \eqref{defk1},
\begin{equation} \label{lawtransitiontime1}
\lim_{r \to \infty} \mathbb{P}_{1_U} \bigg( \frac{\mathcal{T}_G^{Q^0}}{\mathbb{E}_{1_U} [\mathcal{T}_G^{Q^0} \,|\, A_a]} > x  \,\big|\,  A_a \bigg) = \int_x^{\infty} \bigg(\circledast_{k = 1}^{m_{\mathrm{sub}}^a} \mathcal{P}_{\mathrm{sub}}^{f_{k}, S_{\mathrm{sub}}^a}\bigg) (y) dy, \qquad x \in [0,\infty),
\end{equation}
with 
\begin{equation}
\mathcal{P}_{\mathrm{sub}}^{f_k, S_{\mathrm{sub}}^a}(z) =  \frac{S_{\mathrm{sub}}^a}{f_k} \exp\left(-\frac{S_{\mathrm{sub}}^a}{f_k}\,z\right), 
\qquad z \in [0,\infty),
\end{equation} 
and $S_{\mathrm{sub}}^a=\sum_{i \colon\,\bar{d}_{i}=d^*} f_{i}$.

\item[{\rm (III)}] $\beta \in (\frac{1}{d^*-1}, \infty)$: supercritical regime.

\begin{equation} 
\lim_{r \to \infty} \mathbb{P}_{1_U} \bigg( \frac{\mathcal{T}_G^{Q^0}}{\mathbb{E}_{1_U} [\mathcal{T}_G^{Q^0}]}> x \bigg) = \int_x^{\infty} \mathcal{P}_{\mathrm{sup}}(y) \,dy = \left\{\begin{array}{ll}
1, &\text{ if } x \in [0,1), \\[0.2cm]
0, &\text{ if } x \in [1,\infty),
\end{array}
\right.
\qquad x \in [0,\infty),
\end{equation}
with 
\begin{equation}
\mathcal{P}_{\mathrm{sup}}(z) = \delta_1 (z), \qquad z \in [0,\infty),
\end{equation}
where $\delta_1(z)$ is the Dirac function at 1.
\end{itemize}
\end{teorema}

Theorem~\ref{law} will be proved in Section~\ref{subsec:law}. There we will also see why there is no statement for the critical regime (II).


\subsection{Discussion}

Analyzing the transition time for arbitrary bipartite graphs is much harder than for complete bipartite graphs. The key idea is to view the transition time as a sum of subsequent nucleation times for complete bipartite subgraphs. The order in which nodes activate in $V$ is \emph{random}, because it depends on the fluctuations of the activation rates via the queue lengths. However, with high probability the nodes with the least number of active neighbors in $U$ activate first. After each activation, the underlying bipartite graph changes according to which node is activated and which nodes are deactivated. Hence the subsequent activations in $V$ depend on how this graph changes, as well as on the evolution of the network, since the queue lengths (and hence the activation rates) change with time as well. 

To keep track of this evolution, we defined a randomized greedy algorithm in Section~\ref{sec:alg}. If we run the algorithm once, then it generates a specific path of activating nodes in $V$. This is enough to determine the leading order of the transition time as $r \to \infty$, since it only depends on the maximum least degree $d^*$, which is the same for all the admissible paths. Moreover, given $d^*$, we can immediately determine whether we are in the subcritical, the critical or the supercritical regime. If we are interested in the pre-factor of the mean transition time and in its law, then we need to generate all the admissible paths. Theorem~\ref{mostlikelypaths} shows that we can split the mean transition time into a weighted sum over all the admissible paths of the mean nucleation times associated with each activation in the paths. Theorem~\ref{meantransitiontime} gives the mean transition time conditional on the path and shows that the outcome is non-trivial both in the subcritical and the critical regime.
Theorem~\ref{law} gives the law conditional on the path, but fails to capture the critical regime. The reason is that there are intricate dependencies between the subsequent nucleation times along the path.


\section{Nucleation times and queue lengths}
\label{sec:4}

In Section~\ref{subsec:aindep} we introduce the concept of asymptotic independence of forks and we show that in the subcritical and critical regime competing forks can be treated as if they were independent, in the limit as $r \to \infty$ (Proposition~\ref{aindep}). In Section~\ref{subsec:nextnucl} we study the mean and the law of the next nucleation time by using concepts from metastability and results from Section~\ref{subsec:aindep} (Propositions~\ref{meannextnucleationtime} and \ref{lawnextnucleationtime} below). In Section~\ref{subsec:uql} we show how the queue lengths change depending on which node activates in $V$ (Theorem~\ref{qlthm} below). Throughout the section, recall the notation discussed in Remark~\ref{conditionalrmk}.

\subsection{Asymptotic independence of forks} 
\label{subsec:aindep}

In this section we show that, in the limit as $r \to \infty$, forks can be treated as being independent of each other even when they share some nodes. We introduce the concept of \textit{asymptotic independence} of forks, which allows us to treat overlapping forks as if they were independent in the limit as $r \to \infty$. We show that the nucleation time of a fork is asymptotically not influenced by the behavior of other forks sharing nodes with it.

In \cite{BdHNS18} it is shown that, as soon as all the nodes in $U$ of a complete bipartite graph are simultaneously inactive, the first node in $V$ (and subsequently all the others nodes) activate in a very short time interval, negligible compared to the time it takes to deactivate all the nodes in $U$. Hence, the time it takes for the nodes in $U$ to be all simultaneously inactive is the same as the time it takes to activate the first node in $V$, up to an error term that is negligible as $r \to \infty$. In our setting, to study the nucleation times of forks it is enough to study the time it takes to deactivate all their respective nodes in $U$, without considering the set $V$.

\begin{proposizione}{\bf [Asymptotic independence]} 
\label{aindep}
In the subcritical or critical regimes, consider the graph $G_k$, the updated queue lengths $Q^{k-1}$ and the $\bar{d}_k$-fork $W$, where $\bar{d}_k$ is the minimum degree of the nodes in $V_k$. Let $\{u_1, \dots,  u_{\bar{d}_k} \}$ denote the subset of nodes in $U_k$ belonging to fork $W$. For $\alpha < \bar{d}_k$, consider the first time $t = t(\alpha) \in \big[ \sum_{j=1}^{k-1} \bar{\tau}_j, \sum_{j=1}^{k} \bar{\tau}_j \big)$ when $\alpha$ nodes $\{s_1, \dots, s_{\alpha} \}  \subset \{u_1, \dots,  u_{\bar{d}_k} \}$ belonging to fork $W$ are simultaneously inactive. Recall that $T_{W}^{Q^{k-1}}$ is the time it takes $W$ to nucleate, conditional on the updated queue lengths $Q^{k-1}$, and denote by $T_{W_{\alpha}}^{Q(t)}$ the time it takes $W$ to nucleate starting from time $t$ with $\alpha$ nodes inactive, conditional on the updated queue lengths $Q(t)$. The two following statements hold.
\begin{itemize}
\item[(i)] Starting at time $t$ from a state with $\alpha$ nodes inactive,
\begin{equation}
\mathbb{E}_{1_U} [T_{W_{\alpha}}^{Q(t)}] = \mathbb{E}_{1_U} [T_{W}^{Q^{k-1}} ] \, [1+o(1)], \qquad r \to \infty.
\end{equation}
\item[(ii)] Starting at time $t$ from a state with $\alpha$ nodes inactive,
\begin{equation}
\lim_{r \to \infty} \mathbb{P}_{1_U} (T_{W_{\alpha}}^{Q(t)} > x ) = \lim_{r \to \infty} \mathbb{P}_{1_U} (T_{W}^{Q^{k-1}} > x), \qquad x \in [0, \infty).
\end{equation}
\end{itemize}
\end{proposizione}

\begin{proof} We prove the two statements separately.
\begin{itemize}
\item[(i)]
We denote by $\mathcal{S}$ the event that after time $t$ all the nodes of $W$ that are still active become simultaneously inactive before any of the inactive nodes in $ \{s_1, \dots, s_{\alpha} \}$ activates again. We know from Theorem~\ref{teoremapaper1} that the time $T_{\mathcal{S}}$ it takes for $\bar{d}_k - \alpha$ nodes to become simultaneously inactive behaves as an exponential random variable with mean of order $r^{\beta(\bar{d}_k - \alpha-1)}$ as $r \to \infty$, while the time it takes for one of the $\alpha$ inactive nodes to activate is an exponential random variable with mean of order $1/r^{\beta}$. Hence the probability of $\mathcal{S}$ is of order $r^{-\beta(\bar{d}_k - \alpha)} = o(1)$. If $\mathcal{S}$ occurs, then clearly $T_{W_{\alpha}}^{Q(t)}  = T_{\mathcal{S}}$, and hence we have
\begin{equation}
\mathbb{E}_{1_U} [T_{W_{\alpha}}^{Q(t)}  \,|\, \mathcal{S}] = O\big(r^{\beta(\bar{d}_k - \alpha -1)}\big) = o\big(r^{\beta(\bar{d}_k-1)}\big), \qquad r \to \infty.
\end{equation}
On the other hand, if the complementary event $\mathcal{S}^C$ occurs, then the expected time $t'$ for the network to reach the configuration with all the nodes $u_1, \dots, u_{\bar{d}_k}$ active is $o(1)$, and from there it takes time $\mathbb{E}_{1_U} [T_{W}^{Q(t+t')}]$ for $W$ to nucleate. Hence 
\begin{equation}
\mathbb{E}_{1_U} [T_{W_{\alpha}}^{Q(t)} \,|\, \mathcal{S}^C] = o(1)+\mathbb{E}_{1_U} [T_{W}^{Q(t+t')}] = o(1)+\mathbb{E}_{1_U} [T_{W}^{Q^{k-1}}] \, [1+o(1)], \qquad r \to \infty,
\end{equation}
since, in the subcritical or critical regime, the queue lengths at time $t+t'$ are of the same order as the queue lengths at time $\sum_{j=1}^{k-1} \bar{\tau}_j$.
Putting the two complementary events together, we obtain that
\begin{equation}
\begin{split}
\mathbb{E}_{1_U} [T_{W_{\alpha}}^{Q(t)}] & = \mathbb{E}_{1_U} \big[T_{W_{\alpha}}^{Q(t)}   \,|\,\mathcal{S}\big]\, \mathbb{P}_{1_U}(\mathcal{S}) 
+ \mathbb{E}_{1_U} \big[T_{W_{\alpha}}^{Q(t)}  \,|\,  \mathcal{S}^C\big]\, 
\mathbb{P}_{1_U}(\mathcal{S}^C)\\
&= o\big(r^{\beta(\bar{d}_k-1)}\big)o(1) + \big(o(1) + \mathbb{E}_{1_U} [T_{W}^{Q^{k-1}}]\, [1+o(1)]\big)\,[1-o(1)] \\
&= \mathbb{E}_{1_U} [T_{W}^{Q^{k-1}}] \, [1+o(1)], \qquad r \to \infty.
\end{split}
\end{equation}
\item[(ii)]
Using the complementary events $\mathcal{S}$ and $\mathcal{S}^C$, we can write, for all $x \geq 0$,
\begin{equation}
\begin{split}
\lim_{r \to \infty} \mathbb{P}_{1_U} (T_{W_{\alpha}}^{Q(t)} > x )  = & \lim_{r \to \infty} \mathbb{P}_{1_U} (\{T_{W_{\alpha}}^{Q(t)} > x \} \cap \mathcal{S} ) +  \lim_{r \to \infty} \mathbb{P}_{1_U} (\{T_{W_{\alpha}}^{Q(t)} > x \} \cap \mathcal{S}^C  )  \\
=& \lim_{r \to \infty} \mathbb{P}_{1_U} (T_{W}^{Q^{k-1}} > x),
\end{split}
\end{equation}
since $\lim_{r \to \infty} \mathbb{P}_{1_U}(\mathcal{S}) = 0$ and, conditional on $\mathcal{S}^C$, with high probability the network reaches the initial configuration in a negligible time $t'$ after time $t$. Hence it behaves as if at time $t$ all nodes in $U$ were active.
\end{itemize}
\end{proof} 

The above proposition shows that, in the limit as $r \to \infty$, the mean nucleation time of a fork $W$ and its law are not influenced by the fact that some of its nodes are simultaneously inactive at some time. The intuition is that, as $r \to \infty$, the nucleation of a fork is so hard to achieve and takes so long that sharing some nodes with other forks does not help to make the nucleation happen appreciably faster. The network tends to quickly reach the metastable initial configuration with all the nodes in $U$ active, and hence the nucleation time of $W$ can be seen as the time it takes to deactivate all the nodes in $U$ starting from all of them being active. In particular, in case of overlapping forks, the nucleation time of $W$ is not influenced by the behavior of other forks sharing nodes with $W$.

\subsection{Next nucleation time}
\label{subsec:nextnucl}

Given the graph $G_k$, consider the next nucleation time 
\begin{equation}
\bar{\tau}_k = \mathrm{min}_{v \in V_k} \{\mathcal{T}_{v}^{Q^{k-1}} \}
\end{equation} 
from Definition~\ref{mintime}. When the network activates a node, it activates the node that completes the fastest nucleation among the $n_k$ nodes with least degree. We want to find an expression for $\mathbb{E}_{1_U}[\bar{\tau}_k]$. 

In Appendix A we show the computations for the mean next nucleation time in the case when the competing forks are independent of each other. Recall that in the subcritical regime we are considering a minimum of nucleation times that are exponential random variables, while in the critical regime we are considering a minimum of nucleation times that follow a truncated polynomial law (see Theorem~\ref{teoremapaper1}). By using Proposition~\ref{aindep}, we are also able to give explicit asymptotics for the mean next nucleation time without assuming the forks being independent. 

Each nucleation of a fork can be seen as a successful escape from a metastable state, which is represented by the initial configuration where the nodes in $U_k$ in the fork are active and the node in $V_k$ in the fork is inactive. When considering multiple forks, we can view the network as an irreducible Markov process on a state space $\Omega$. The first nucleation can be described by a regenerative process where the Markov process leaves a metastable state $x_0$ (with all the nodes in $U_k$ active) and reaches a set $S$, which represents the set of states where at least one of the forks of minimum degree has all the nodes in $U$ simultaneously inactive. The set $S$ is rare for the Markov process, in the sense that the fraction of time spent in $S$ is small. Indeed, once the process reaches state $S$, with high probability the first nucleation happens in time $o(1)$ after one of the nodes of minimum degree in $V_k$ activates. Denote by $T^k_{x_0 \rightarrow S}$ the time it takes to go from $x_0$ to $S$, and note that 
\begin{equation}
\label{hittingtimenucleation}
\mathbb{E}_{1_U}[\bar{\tau}_k] = \mathbb{E}_{1_U}[T^k_{x_0 \rightarrow S}]\, [1+o(1)], \qquad r \to \infty.
\end{equation}

\begin{lemma}{\bf [Mean return time to metastable state]} \label{returntime}
For $k = 1, \dots, N$, suppose that $k-1$ nodes in $V$ have already been activated. Then, with high probability the time $R_{U_k}^x$ it takes for the network $G_k$ to reach the configuration with all the nodes in $U_k$ active (the metastable state $x_0$) starting from any other configuration $x$ is negligible, i.e., with high probability
\begin{equation}
\mathbb{E}_{1_U} [R_{U_k}^x] = o(1), \qquad r \to \infty.
\end{equation}
In particular, let $R_{U_k}^{k-1}$ be the time it takes for the network $G_k$ to reach the configuration with all the nodes in $U_k$ active starting from the moment the $(k-1)$-th node in $V$ activated. Then, with high probability
\begin{equation}
\mathbb{E}_{1_U} [R_{U_k}^{k-1}] = o(1), \qquad r \to \infty.
\end{equation}
\end{lemma}
\begin{proof}
Recall that at any time $t$, the activation and deactivation of each node $u \in U$ are described by random variables with rates $g_U(Q_u(t))$ and $1$, respectively. Hence, each node $u \in U_k$ takes on average one unit of time to deactivate and $1/g_U(Q_u(t))$ to activate. Since in the subcritical and critical regime the queue lengths at any node at any moment are of order $r$ (see Section~\ref{subsec:uql} for more details), we can say that $1/g_U(Q_u(t)) = o(1)$. Suppose that, at some time $t$, node $u \in U_k$ is active and node $u' \in U_k$ is inactive, i.e., $X_u(t) = 1$ and $X_{u'}(t) = 0$. Since
\begin{equation}
\lim_{r \to \infty} \mathbb{P}_{1_U} (u' \text{ activates} < u \text{ deactivates}) = 1,
\end{equation}
and there is a finite number of nodes in $U_k$, with high probability, starting from any configuration $x$, all the nodes in $U_k$ will be active on average in $o(1)$. Hence, as $r \to \infty$, $\mathbb{E}_{1_U} [R_{U_k}^x] = o(1)$, and in particular $\mathbb{E}_{1_U} [R_{U_k}^{k-1}] = o(1)$.
\end{proof}

We are now ready to state a result for the mean next nucleation time in the subcritical and the critical regime.

\begin{proposizione}{\bf [Mean next nucleation time]} \label{meannextnucleationtime}
Consider the graph $G_k$. Recall that $\bar{d}_k$ is the minimum degree of a node in $V_k$ and $n_k$ is the number of forks of degree $\bar{d}_k$ in $G_k$.
\begin{itemize}
\item[{\rm (I)}] $\beta \in \big(0, \frac{1}{\bar{d}_k-1}\big)$: subcritical regime. The mean next nucleation time satisfies, for any $u \in U_k$,
\begin{equation} \label{meansubcriticaltau}
\mathbb{E}_{1_U} [\bar{\tau}_k] = f_k \, \mathbb{E}_{1_U}[\mathcal{T}_{v_k^*}^{Q^{k-1}}] = f_k\, F^{k}_{\mathrm{sub}}\, (Q_u^{k-1})^{\beta(\bar{d}_k -1)} \,[1+o(1)],\qquad r \to \infty,
\end{equation}
with 
\begin{equation}  \label{fksubcritical}
f_k =\frac{1}{n_k}.
\end{equation}
\item[{\rm (II)}] $\beta = \frac{1}{\bar{d}_k-1}$: critical regime. The mean next nucleation time satisfies, for any $u \in U_k$,
\begin{equation} \label{meancriticaltau}
\mathbb{E}_{1_U} [\bar{\tau}_k] = f_k \,\mathbb{E}_{1_U}[\mathcal{T}_{v_k^*}^{Q^{k-1}}]  = f_k\,  F^{k}_{\textrm{cr}} \, Q_u^{k-1} \,[1+o(1)],\qquad r \to \infty,
\end{equation}
with
\begin{equation} \label{fkcritical}
f_k = \frac{ \bar{d}_k B^{-(\bar{d}_k -1)} + c-\rho_U}{n_k \bar{d}_k B^{-(\bar{d}_k -1)} + c-\rho_U}.
\end{equation}
\end{itemize}
\end{proposizione}

\begin{proof}
By Proposition~\ref{aindep}, in the limit as $r \to \infty$ we may consider arbitrarily overlapping forks as if they were independent of each other. Therefore the computations for the mean next nucleation time carried out in Appendix~\ref{appA} for the case of independent forks can be used for the case of overlapping forks as well. For completeness, in the subcritical regime (I) we offer a proof that uses a different argument, which cannot be used in the critical regime (II) because the queues are changing on scale $r$ over time. 

Consider the stationary distribution $\pi$ of the Markov process mentioned above and recall that $x_0$ represents the metastable state with all the nodes in $U_k$ active. For any $u \in U_k$ the probability of the set $S$ is given by 
\begin{equation} \label{Sj}
\pi(S) =  \sum_{j=1}^{n_k} \pi(S_j) \,  [1+o(1)] 
= n_k \bigg(\frac{1}{B\, (Q_u^{k-1})^{\beta}}\bigg)^{\bar{d}_k} \,  [1+o(1)] , \qquad r \to \infty,
\end{equation}
where $S_j$ is the state in which the $j$-th fork has all its nodes simultaneously inactive. The terms representing multiple forks with all their nodes simultaneously inactive contribute in a negligible way to $\pi(S)$. Moreover, for $j = 1, \dots, n_k$, since the stationary distribution $\pi(S_j)$ can be interpreted as the long run proportion of time spent in state $S_j$, for any $u \in U_k$ we can write
\begin{equation}
\begin{split}
\pi(S_j) &=  \frac{\mathbb{E}_{1_U}[\text{time spent in } S_j ]}{\mathbb{E}_{1_U}[\text{time spent in } S_j ] 
+ \mathbb{E}_{1_U} [T^k_{x_0\rightarrow S_j}]}  \,  [1+o(1)] \\
& =  \frac{ \frac{1}{\bar{d}_k} \frac{1}{B\, (Q_u^{k-1})^{\beta}}}{  \frac{1}{\bar{d}_k} 
\frac{1}{B\, (Q_u^{k-1})^{\beta}} 
+  F_{\mathrm{sub}}^k \, (Q_u^{k-1})^{\beta (\bar{d}_k -1)}  } \,  [1+o(1)] \\
&=  \frac{ \frac{1}{\bar{d}_k} \frac{1}{B\, (Q_u^{k-1})^{\beta}} }
{ F_{\mathrm{sub}}^k \, (Q_u^{k-1})^{\beta (\bar{d}_k -1)}  } \, [1+o(1)] 
=  \bigg(\frac{1}{B \, (Q_u^{k-1})^{\beta}}\bigg)^{\bar{d}_k} \, [1+o(1)], \qquad r \to \infty.
\end{split}
\end{equation}
This proves \eqref{Sj}. 

Using the same type of argument, we can compute $\mathbb{E}_{1_U} [T_{x_0 \rightarrow S}]$. Indeed, for any $u \in U_k$,
\begin{equation}
\begin{aligned}
\pi(S) &=  \frac{\mathbb{E}_{1_U}[\text{time spent in } S ]}{\mathbb{E}_{1_U}[\text{time spent in } S ] 
+ \mathbb{E}_{1_U} [T^k_{x_0 \rightarrow S}]} \\ 
&= 
\frac{ \frac{1}{\bar{d}_k} \frac{1}{B \, (Q_u^{k-1})^{\beta}} }
{  \frac{1}{\bar{d}_k} \frac{1}{B \, (Q_u^{k-1})^{\beta}} 
+  \mathbb{E}_{1_U} [T^k_{x_0\rightarrow S}]  } \,  [1+o(1)] 
= \frac{ \frac{1}{\bar{d}_k} \frac{1}{B \, (Q_u^{k-1})^{\beta}} }
{  \mathbb{E}_{1_U} [T^k_{x_0\rightarrow S}] } \,  [1+o(1)] ,
\qquad r \to \infty,
\end{aligned}
\end{equation}
hence, by inverting it, we get
\begin{equation} \label{termsx}
\begin{split}
\mathbb{E}_{1_U} [T^k_{x_0\rightarrow S}] & = \frac{\frac{1}{\bar{d}_k} \frac{1}{B\, (Q_u^{k-1})^{\beta}} }{\pi(S)} \, [1+o(1)]
= \frac{\frac{1}{\bar{d}_k} \frac{1}{B\, (Q_u^{k-1})^{\beta}} }
{n_k \left(\frac{1}{B \, (Q_u^{k-1})^{\beta}}\right)^{\bar{d}_k}}\, 
[1+o(1)] \\
&= f_k\, F^{k}_{\mathrm{sub}}\, (Q_u^{k-1})^{\beta(\bar{d}_k -1)} \,[1+o(1)],\qquad r \to \infty,
\end{split}
\end{equation}
with
\begin{equation}
f_k = \frac{1}{n_k}.
\end{equation}
The proof is completed by using \eqref{hittingtimenucleation}.
\end{proof}

\begin{corollario}{\bf [Pre-factor adjustment]} \label{pk}
Given the graph $G_k$, conditional on the next activating node of degree $\bar{d}_k$,
\begin{equation}
\mathbb{E}_{1_U}[\bar{\tau}_k \,|\,  Y_k = i_k] = \mathbb{E}_{1_U}\Big[\mathrm{min}_{v \in V_k} \mathcal{T}_{v}^{Q^{k-1}} 
\,\big|\, Y_k = i_k\Big] = f_k \,\mathbb{E}_{1_U}[\mathcal{T}_{v_{i_k}}^{Q^{k-1}}], \qquad r \to \infty,
\end{equation}
where $f_k$ is as in \eqref{fksubcritical} or \eqref{fkcritical} when a subcritical node or a critical node activates, respectively.
\end{corollario}

\begin{proof}
The claim follows from Proposition~\ref{meannextnucleationtime}.
\end{proof}

In the subcritical regime (I), the queue lengths do not change on scale $r$ and therefore the renewal theory developed in \cite{FMNS15} applies, which is tailored to exponential behavior in metastable regimes. In the critical regime (II), however, the queue lengths do change on scale $r$ and \cite{FMNS15} does not apply. For details, see Section~\ref{subsec:uql}.  

Recall that $\Omega$ is the state space of the Markov process and that, in our notation, $\bar{\tau}_k = T_{x_0\rightarrow S}$.

\begin{definizione}{\bf [Recurrence property]} 
Let $H > 0$ and $h \in (0,1)$. We say that the pair $(x_0, S)$ satisfies $\mathrm{Rec}(H,h)$ if
\begin{equation} \label{rec}
\sup_{x \in \Omega} \mathbb{P} \big( T_{x \rightarrow \{x_0, S\}} > H \big) \leq h.
\end{equation}
\end{definizione} 

\begin{proposizione}{\bf [Law of the next nucleation time in the subcritical regime \cite[Theorem 2.3]{FMNS15}]} \label{lawnextnucleationtime}
Consider the pair $(x_0,S)$ such that $\mathrm{Rec}(H,h)$ holds for $0<H < \mathbb{E}_{1_U} [\bar{\tau}_k]$, with $\epsilon = H/\mathbb{E}_{1_U} [\bar{\tau}_k]$ and $h$ sufficiently small. Then there exist functions $C(\epsilon, h)$ and $\lambda(\epsilon, h)$, satisfying $C(\epsilon, h), \lambda(\epsilon, h) \to 0$ as $\epsilon, h \downarrow 0$, such that, for any $t > 0$, 
\begin{equation}
\bigg| \, \mathbb{P} \bigg( \frac{\bar{\tau}_k}{\mathbb{E}_u [\bar{\tau}_k]} > t \bigg) - e^{-t} \, \bigg| \leq C e^{-(1-\lambda)t}.
\end{equation}
\end{proposizione}

\begin{proof}
We choose $H$ to be a constant, and without loss of generality set $H=1$. We claim that the pair $(x_0, S)$ satisfies the property $\mathrm{Rec}(H,h)$ with $h$ sufficiently small. Indeed, starting from any configuration $x \in \Omega$, the network reaches the set $\{x_0, S\}$ in a small time which is $o(1)$.

If the starting configuration $x$ is one of the configurations $S_j$, $j=1, \dots, n_k$, corresponding to the set $S$, then we are done. Otherwise, by Lemma~\ref{returntime}, the metastable state $x_0$ attracts in time $o(1)$ every configuration $x$ for which some forks have some nodes in $U$ inactive. It is therefore immediate that $T_{x \rightarrow \{x_0, S\}}$ is smaller than $H$ with high probability, which is what we need in order to claim that \eqref{rec} holds when $h$ is sufficiently small. Note that we can let $h \downarrow 0$ as  $r \to \infty$.

We recover from Proposition~\ref{meannextnucleationtime} that the ratio between $H$ and the mean next nucleation time is sufficiently small. Indeed, $\epsilon = H/\mathbb{E}_{1_U} [\bar{\tau}_k] \downarrow 0$ as $r \to \infty$. Hence a straightforward application of \cite[Theorem 2.3]{FMNS15} allows us to conclude that the law of the next nucleation time divided by its mean is exponential with unit rate as $r \to \infty$. 
\end{proof}

\subsection{Updated queue lengths}
\label{subsec:uql}

In this section we analyze in more detail how the mean queue lengths change over time and how they affect the mean nucleation times associated with each step of the algorithm. Since the queue lengths have a \textit{good behavior} (see Remark~\ref{gbremark}), we will often approximate them by their mean, or vice versa, at the cost of an error term that is negligible as $r \to \infty$. Note that in case of activation of a supercritical node, the queue lengths become of order less than $r$, but at that point we do not need any more control on their behavior since we know how the transition occurs.

We start with initial queue lengths $Q^0 = (Q_U^0, Q_V^0)$ where $Q_U(0)=\gamma_U r$ and $Q_V(0) = \gamma_V r$, with $\gamma_U > \gamma_V \geq 0$. We are interested in studying how the queue lengths change along a fixed path, depending on which types of forks we encounter at each activation. Fix an admissible path and consider the sequence of nodes activating in $V$. 

Similarly to \eqref{nt&q}, the next nucleation time $\bar{\tau}_k = \min_{v \in V_k} \mathcal{T}_v^{Q^{k-1}}$ (recall Definition~\ref{mintime}) satisfies
\begin{equation} \label{mnnt&ql}
\mathbb{E}_{1_U} [\bar{\tau}_k] = f'_k \, r^{1 \wedge \beta(\bar{d}_k-1)}\,[1+o(1)], \qquad r \to \infty,
\end{equation}
where $f'_k$ depends on $f_k$, on the constants $F^{k}_{\textrm{sub}}, F^{k}_{\textrm{cr}}, F^{k}_{\textrm{sup}}$ (for the three regimes, respectively), and on the updated queue lengths. The following theorem shows how the queue lengths change according to which type of node activates in $V$. 
  
Recall the notation $o(r^{\alpha}) = o_{\mathbb{P}_{1_U}}(r^{\alpha})$ for $\alpha \geq 0$, which refers to a random variable determined by the law $\mathbb{P}_{1_U}$ that goes to 0 in distribution when divided by $r^{\alpha}$ as $r \to \infty$.

\begin{teorema}{\bf [Mean updated queue length]} \label{qlthm}
Let $(\bar{d}_k)_{k=1}^N$ be the sequence of degrees in a fixed admissible path and $d^* = \max_{1 \leq k \leq N} \bar{d}_k$. \begin{itemize} 
\item[{\rm (I)}] 
$\beta \in (0, \frac{1}{d^*-1})$: subcritical regime. After step $k$, the mean queue length at any node $u \in U$ is
\begin{equation} \label{qlsubcritical}
\mathbb{E}_{1_U} [Q_u^{k}] = \gamma_U r \,[1 + o(1)], \qquad r \to \infty.
\end{equation}
\item[{\rm (II)}] 
$\beta = \frac{1}{d^*-1}$: critical regime. After step $k$, the mean queue length at any node $u \in U$ is
\begin{equation} \label{qlcritical}
\mathbb{E}_{1_U} [Q_u^{k}] =\gamma_U^{(k)} r \,[1 + o(1)], \qquad r \to \infty,
\end{equation}
with 
\begin{equation} \label{qlcriticalrv}
\gamma_U^{(k)}= \gamma_U - (c-\rho_U) \sum_{\substack{1 \leq i \leq k \\ i:\, \bar{d}_i = d^*}} f'_i > 0,
\end{equation}
where for a critical node $v_i$ the coefficient $f'_i$ is defined in a recursive way as 
\begin{equation} \label{recursivef'}
f_i' = \frac{1}{n_{i} \bar{d}_{i} B^{-(\bar{d}_{i}-1)} + c-\rho_U} \, \bigg(\gamma_U - (c-\rho_U) \sum_{\substack{1 \leq j \leq i-1 \\ j:\, \bar{d}_j = d^*}} f_j'\bigg) > 0,
\end{equation}
\begin{equation} 
f'_1 = 
\left\{\begin{array}{ll} 
\frac{1}{n_1} \,\frac{1}{\bar{d}_1 B^{-(\bar{d}_1 -1)}} \,\gamma_U, & \text{ if } \bar{d}_1 < d^*, \\[0.2cm]
\frac{1}{n_1\bar{d}_1 B^{-(\bar{d}_1 -1)} + c-\rho_U} \, \gamma_U, & \text{ if } \bar{d_1} = d^*.
\end{array} \right.
\end{equation}
\item[{\rm (III)}] 
$\beta \in (\frac{1}{d^*-1}, \infty)$: supercritical regime. After step $k$, the mean queue length at any node $u \in U$, if any supercritical node in $V$ has activated, is
\begin{equation} \label{qlsupercritical}
\mathbb{E}_{1_U} [Q_u^{k}] = o(r), \qquad r \to \infty.
\end{equation}
\end{itemize}
\end{teorema}

\begin{proof}
We treat the three regimes separately.

\medskip\noindent
(I) $\beta \in (0, \frac{1}{d^*-1})$. All the nodes in $V$ are subcritical, in particular the first node $v_1 \in V$. Then $\mathbb{E}_{1_U}[{\bar{\tau}_1}] = o(r)$ as $r \to \infty$. The mean queue length at any node $u \in U$ after node $v_1$ activates is (recall Section~\ref{ss:model})
\begin{equation}
\begin{split}
\mathbb{E}_{1_U}[Q_u(\bar{\tau}_1)] 
&= \mathbb{E}_{1_U}[\gamma_U r - (c-\rho_U) {\bar{\tau}_1}] =  \gamma_U r - (c-\rho_U) \mathbb{E}_{1_U}[{\bar{\tau}_1}] \\
& = \gamma_U r \, [1+o(1)], \qquad r \to \infty,
\end{split}
\end{equation}
which means that after the first activation the mean queue lengths are the same as before, up to an error term $o(1)$. Iterating this reasoning, we conclude that the mean queue lengths remain approximately the same as long as we activate subcritical nodes in $V$.

\medskip\noindent
(II) $\beta = \frac{1}{d^*-1}$. If the first node $v_1 \in V$ is subcritical, then the time it takes to nucleate its fork does not influence the mean queue lengths by much, as seen in (I). Without loss of generality, we may therefore assume that $v_1$ is critical. Then $\mathbb{E}_{1_U}[{\bar{\tau}_1}] = f'_1 r $ is of order $r$. The mean queue length at any node $u \in U$ after node $v_1$ activates is
\begin{equation}
\begin{split}
\mathbb{E}_{1_U}[Q_u(\bar{\tau}_1)] 
&= \mathbb{E}_{1_U}[\gamma_U r - (c-\rho_U)\bar{\tau}_1]= \gamma_U r - (c-\rho_U) \mathbb{E}_{1_U}[{\bar{\tau}_1}] \\
&= (\gamma_U - (c-\rho_U)f'_1)r\,[1+o(1)] = \gamma_U^{(1)}r\,[1+o(1)], \qquad r \to \infty,
\end{split}
\end{equation}
where $\gamma_U^{(1)} = \gamma_U - (c-\rho_U)f'_1 > 0$. 

If the second node $v_2 \in V$ is subcritical, then again the time it takes to nucleate its fork does not influence the mean queue lengths by much. Assume therefore that $v_2$ is critical. Then the fork requires a nucleation time of order $r$, namely, $\mathbb{E}_{1_U}[{\bar{\tau}_2}] = f'_2 r$. The mean queue length at any node $u \in U$ after node $v_2 \in V$ activates is
\begin{equation}
\begin{split}
\mathbb{E}_{1_U}[Q_u(\bar{\tau}_1 + \bar{\tau}_2)] 
&= \mathbb{E}_{1_U}[\gamma_U r - (c-\rho_U) (\bar{\tau}_1+\bar{\tau}_2)] =  \gamma_U r - (c-\rho_U)(\mathbb{E}_{1_U}[{\bar{\tau}_1}] +\mathbb{E}_{1_U}[{\bar{\tau}_2}] ) \\
& = (\gamma_U - (c-\rho_U)(f'_1 + f'_2))r \, [1+o(1)] = \gamma_U^{(2)}r \, [1+o(1)], \qquad r \to \infty,
\end{split}
\end{equation}
where $\gamma_U^{(2)}= \gamma_U - (c-\rho_U)(f'_1 + f'_2) > 0$. 

More generally, we have that, for any node $u \in U$,
\begin{equation}
\mathbb{E}_{1_U}[Q_u^{k}] = \gamma_U^{(k)} r \,[1 + o(1)], \qquad r \to \infty,
\end{equation}
with 
\begin{equation}
\gamma_U^{(k)}= \gamma_U - (c-\rho_U) \sum_{\substack{1 \leq i \leq k \\ i:\, \bar{d}_i = d^*}} f'_i > 0,
\end{equation}
where the last sum is over all the critical nodes activated up to step $k$. Each of them contributes with a positive coefficient $f_i'$ which is given by the recursive relation
\begin{equation} \label{f'def}
\begin{split}
f_i' & = f_{i} \, F_{\mathrm{cr}}^{i} \, \gamma_U^{(i-1)} \\
& = \frac{\bar{d}_{i} B^{-(\bar{d}_{i}-1)}+ c-\rho_U}{n_{i} \bar{d}_{i} B^{-(\bar{d}_{i}-1)} + c-\rho_U}\, \frac{1}{\bar{d}_{i} B^{-(\bar{d}_{i}-1)}+ c-\rho_U} \, \gamma_U^{(i-1)} \\
& = \frac{1}{n_{i} \bar{d}_{i} B^{-(\bar{d}_{i}-1)} + c-\rho_U} \, \Bigg(\gamma_U - (c-\rho_U) \sum_{\substack{1 \leq j \leq i-1 \\ j:\, \bar{d}_j = d^*}} f_j'\Bigg).
\end{split}
\end{equation}
Note that the coefficients $f_k'$ introduced in \eqref{mnnt&ql} are defined for every $k = 1, \dots, N$, but in the above computations we are only interested in the ones associated with the critical nodes. Note that
\begin{equation} 
f'_1 = 
\left\{\begin{array}{ll} 
\frac{1}{n_1} \,\frac{1}{\bar{d}_1 B^{-(\bar{d}_1 -1)}} \,\gamma_U, & \text{ if } \bar{d}_1 < d^*, \\[0.2cm]
\frac{1}{n_1\bar{d}_1 B^{-(\bar{d}_1 -1)} + c-\rho_U} \, \gamma_U, & \text{ if } \bar{d_1} = d^*.
\end{array} \right.
\end{equation}

\medskip\noindent
(III) $\beta \in (\frac{1}{d^*-1}, \infty)$. If the first node $v_1 \in V$ is subcritical, then its nucleation time does not influence the mean queue lengths by much, as seen in (I). If $v_1$ is critical, then the mean queue lengths decrease but remain of order $r$, as seen in (II). We therefore assume that $v_1$ is supercritical. Then $\mathbb{E}_{1_U}[{\bar{\tau}_1}] = T_U(r) = \frac{\gamma_U}{c-\rho_U} r\,[1+o(1)]$, as $r \to \infty.$ Indeed, from Theorem~\ref{teoremapaper1} we know that the mean nucleation time of a supercritical fork is given by the expected time it takes for the queue length to hit zero. This holds for every supercritical node in $V$ and therefore it is true also for $\mathbb{E}_{1_U}[{\bar{\tau}_1}]$. Hence, the mean queue length at any node $u \in U$ after node $v_1 \in V$ activates is
\begin{equation}
\mathbb{E}_{1_U}[Q_u(\bar{\tau}_1)] = \mathbb{E}_{1_U}[\gamma_U r - (c-\rho_U) \bar{\tau}_1] 
= \gamma_U r - (c-\rho_U) \mathbb{E}_{1_U}[\bar{\tau}_1]= o(r), \qquad r \to \infty.
\end{equation}
More generally, the mean queue lengths become $o(r)$ as soon as the first supercritical node is activated, independently of what was activated before. Thus, after any step $k$ the mean queue length at any node $u \in U$, if any supercritical node has activated, is
\begin{equation}
\mathbb{E}_{1_U}[Q_u^{k}] = o(r), \qquad r \to \infty.
\end{equation}
\end{proof}

In summary, we have shown that if we activate a subcritical node, then we do not change the mean queue lengths at nodes in $U$ by much: they only decrease by a factor $o(1)$. On the other hand, if we activate a critical node, then the mean queue lengths drop significantly, but still remain of order $r$. Finally, if we activate a supercritical node, then the mean queue lengths become $o(r)$, and remain so during all the successive nucleations. Recall that by Remark~\ref{gbremark} we can approximate the queue lengths with their mean, hence with the help of \eqref{mnnt&ql} we know how to relate the mean next nucleation times of the forks to the updated queue lengths after each activation. Hence we know that, once we activate a node that contributes order $r$ to the total mean transition time, we can ignore the contribution of all the previous and all the subsequent subcritical nodes. Once we activate a supercritical node, we can ignore the contribution of all the subsequent nodes, since their queue lengths are $o(r)$.


\section{Analysis of the algorithm}
\label{sec:algprop}

In Section~\ref{subsec:rec} we describe how the algorithm acts on an arbitrary bipartite graph. (In Section~\ref{subsec:example} we already illustrated this via an example.) In Section~\ref{subsec:grconst} we prove the greediness and the consistency of the algorithm. In Section~\ref{subsec:complexity} we discuss the algorithm complexity.


\subsection{Recursion}
\label{subsec:rec}

Consider the graph $G= G_1= ((U_1,V_1),E_1)$. The first node activating in $V_1$ is the one with the least degree, since this requires the least number of nodes in $U_1$ to become simultaneously inactive. Since the expected time until $m$ nodes in $U_1$ are simultaneously inactive is of order $r^{1 \wedge \beta(m-1)}$, the first node to activate in $V_1$ is with high probability $v_{Y_1}$ such that $d(v_{Y_1}) = \bar{d}_1 = \min_{v \in V_1} d(v)$, where $d(v)$ denotes the degree of node $v$ in the graph $G_1$. We make the algorithm pick as first node a node $v_{Y_1}$ with least degree in $V_1$.  If there are multiple nodes with the same least degree, then the algorithm chooses one of them uniformly at random. If the least degree $\bar{d}_1$ is such that $\beta(\bar{d}_1 -1) > 1$, then the algorithm chooses a node uniformly at random among all nodes in $V_1$. Let $G_1'(U_1',V_1')$ be  the complete bipartite subgraph of $G_1$ with $U_1' =  \{ u \in U_1 : uv_{Y_1} \text{ is an edge of } G_1 \}$ and $V_1' =  \{ v_{Y_1} \}$.  According to Theorem~\ref{teoremapaper1}, the associated nucleation time $\mathcal{T}_{v_{Y_1}}^{Q^0}$ satisfies, for any $u \in U_1'$,
\begin{equation}
\mathbb{E}_{1_U} [\mathcal{T}_{v_{Y_1}}^{Q^0}] 
= F^1 \, (Q_u^0)^{1 \wedge \beta(\bar{d}_1-1)}\, [1+o(1)],  \qquad r \to \infty.
\end{equation}

Reasoning as above, we see that the algorithm picks as second node a node $v_{Y_2}$ with the least number of active neighbors left in $G$. Consider the bipartite graph $G_2 = ((U_2, V_2), E_2)$ with $U_2 = U_1 \setminus U_1'$, $V_2 = V_1 \setminus V_1' = V_1 \setminus \{v_{Y_1} \}$ and $E_2 = \{uv: u \in U_2, v \in V_2 \}$. If we denote by $d_2(v)$ the degree of a node $v \in V_2$ in $G_2$, then $v_{Y_2}$ is such that $d_2(v_{Y_2}) = \bar{d}_2 = \min_{v \in V_2} d_2(v)$. If there are multiple nodes with the same least degree, then the algorithm again chooses one uniformly at random. If the least degree $\bar{d}_2$ is such that $\beta(\bar{d}_2 -1) > 1$, then we choose a node uniformly at random among all nodes in $V_2$. Let $G_2'(U_2',V_2')$ be the complete bipartite subgraph of $G$ with $U_2' =  \{u \in U_2\colon\, uv_{Y_2} \text{ is an edge of } G_2 \}$ and $V_2' =  \{ v_{Y_2} \}$. The associated nucleation time $ \mathcal{T}_{v_{Y_2}}^{Q^1}$ satisfies, for any $u \in U_2'$,
\begin{equation}
\mathbb{E}_{1_U} [\mathcal{T}_{v_{Y_2}}^{Q^1}] = F^2 \, (Q_u^1)^{1 \wedge \beta(\bar{d}_2-1)}\, [1+o(1)],  
\qquad r \to \infty.
\end{equation}

Iterating this procedure until all the nodes in $V_1$ are active, we find an admissible path. Note that, depending on the choice the algorithm makes at each step, there may be multiple admissible paths.


\subsection{Greediness and consistency}
\label{subsec:grconst}

We first prove Lemma~\ref{lemmaab}. After that we prove Propositions~\ref{greedinessprop} and \ref{consistencyprop}.

\begin{proof}[Proof of Lemma~\ref{lemmaab}]
The proof is by contradiction. Suppose that $d_a^* > d_b^*$. Denote by $d_{k,a}(v)$ and $d_{k,b}(v)$ the degrees of node $v \in V_k$ at step $k=1, \dots, N$ in paths $a$ and $b$, respectively.

Consider the node $w_1 \in V$ such that, at some step $k_1^a$ in path $a$, $d_{k_1^a,a}(w_1) = \bar{d}_{k_1^a,a} = d_a^*$. Then $d(w_1) \geq d_a^*$ in $G$. On the other hand, in path $b$, when $w_1$ is activated at some step $k_1^b$, it has degree $d_{k_1^b,b}(w_1) \leq d_b^*$. This implies that some of the edges of $w_1$ (at least $d_a^* - d_b^*$ edges) have already been processed via previous forks in path $b$. At least one of these forks must have nucleated before the fork of $w_1$, in path $b$ but not in path $a$, say, the fork of $w_2$. Hence there exists a node $w_2 \in V$ such that, at some step $k_2^b < k_1^b$ in path $b$, $d_{k_2^b,b}(w_2) \leq d_b^*$. This node has not yet been activated at step $k_1^a$ in path $a$, so it must be that $d_{k_1^a,a}(w_2) \geq d_a^*$, otherwise the algorithm would choose node $w_2$ before node $w_1$. Say that node $w_2$ will be activated at step $k_2^a > k_1^a$ in path $a$. Then, $d(w_2) \geq d_a^*$ in $G$. As before, this implies that some of its edges have already been processed with previous forks in path $b$. Again, at least one of these forks must have nucleated before the fork of $w_2$, in path $b$ but not in path $a$, say, the fork of $w_3$. Hence there exists a node $w_3 \in V$ such that, at some step $k_3^b < k_2^b$ in path $b$, $d_{k_3^b,b}(w_3) \leq d_b^*$. This node has not yet been activated at step $k_2^a$ in path $a$, nor at step $k_1^a$, so $d_{k_1^a,a}(w_3) \geq d_{k_1^a,a}(w_1)\geq d_a^*$, otherwise the algorithm would choose node $w_3$ before node $w_1$. Hence $d(w_3) \geq d_a^*$ in $G$. 

We can iterate this argument. Since there are only $N$ nodes in $V$, we get a contradiction after we have considered all the nodes. 
\end{proof}

We are now able to prove the greediness and the consistency of the algorithm.

\begin{proof}[Proof of Proposition~\ref{greedinessprop}]
By Lemma~\ref{lemmaab}, we know that the maximum least degree of an admissible path is the smallest possible. We know that the order of the mean transition time along a path is related to $d^*$ and depends on the value of $\beta$. Hence, Lemma~\ref{lemmaab} implies that the mean transition time along an admissible path is the shortest possible, in the sense that it has the smallest order of $r$ possible.
\end{proof}

\begin{proof}[Proof of Proposition~\ref{consistencyprop}]
Lemma~\ref{lemmaab} proves equality for any two admissible paths. This leads to the same order of the mean transition time.
\end{proof}

Despite the fact that $d^*$ does not depend on which admissible path the algorithm generates, its multiplicity does. Fig.~\ref{fig:multi} shows a graph on which the algorithm can generate two different paths with the same maximum least degree but with different multiplicities. 

\begin{figure}[htbp]
\begin{center}
\begin{tikzpicture}[scale = 0.9]
\draw[fill] (0,0) circle (0.1);
\draw[fill] (0,0.6) circle (0.1);
\draw[fill] (0,1.2) circle (0.1);
\draw[fill] (0,1.8) circle (0.1);
\draw[fill] (0,2.4) circle (0.1);
\draw[fill] (0,3) circle (0.1);
\draw[fill] (0,3.6) circle (0.1);
\draw[fill] (3,0.6) circle (0.1);
\draw[fill] (3,1.8) circle (0.1);
\draw[fill] (3,3) circle (0.1);
\node [left] at (-0.2,0) {$u_7$};
\node [left] at (-0.2,0.6) {$u_6$};
\node [left] at (-0.2,1.2) {$u_5$};
\node [left] at (-0.2,1.8) {$u_4$};
\node [left] at (-0.2,2.4) {$u_3$};
\node [left] at (-0.2,3) {$u_2$};
\node [left] at (-0.2,3.6) {$u_1$};
\node [right] at (3.2,0.6) {$v_3$};
\node [right] at (3.2,1.8) {$v_2$};
\node [right] at (3.2,3) {$v_1$};
\draw[thick] (0.2,3.55) -- (2.8,3.05);
\draw[thick] (0.2,3) -- (2.8,3);
\draw[thick] (0.2,2.45) -- (2.8,2.95);
\draw[thick] (0.2,3.5) -- (2.8,1.9);
\draw[thick] (0.2,2.95) -- (2.8,1.85);
\draw[thick] (0.2,2.4) -- (2.8,1.8);
\draw[thick] (0.2,1.8) -- (2.8,1.75);
\draw[thick] (0.2,1.25) -- (2.8,1.7);
\draw[thick] (0.2,1.15) -- (2.8,0.65);
\draw[thick] (0.2,0.6) -- (2.8,0.6);
\draw[thick] (0.2,0.05) -- (2.8,0.55);
\end{tikzpicture}
\end{center}
\caption{\small The algorithm may generate the path $v_1, v_2, v_3$ or the path $v_3, v_1, v_2$ with different 
multiplicity of $d^*$.}
\label{fig:multi}
\end{figure}
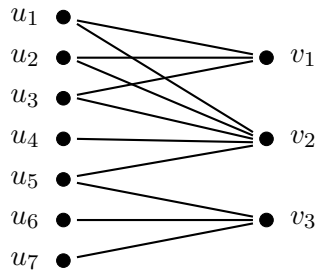

\subsection{Algorithm complexity}
\label{subsec:complexity}

The randomized greedy algorithm we constructed can be implemented in different ways according to what we want to compute.

\begin{itemize}
\item 
In order to know the leading order of the mean transition time as $r \to \infty$, it is enough to recover the maximum least degree $d^*$ from the graph. By Proposition~\ref{consistencyprop} we know that $d^*$ is the same for all the admissible paths. Hence it is enough to run the algorithm once and, by comparing the value of $d^*$ with the value of $\beta$, we are able to determine whether we are in the subcritical, the critical or the supercritical regime.

In this case the \emph{computational complexity of the algorithm is polynomial} in the number of nodes in $V$, and so the leading order of the mean transition time is quickly determined. More precisely, the algorithm has a complexity of $\mathcal{O}(|U| |V|^2)$.
\item 
If we are interested in the precise asymptotics of the mean transition time and in its law as $r \to \infty$, then we need to compute the pre-factor of the leading order term. To do so, we need to run the algorithm multiple times, until all the admissible paths are generated, in order to recover all the possible sequences $(\bar{d}_k)_{k=1}^N$ and $(n_k)_{k=1}^N$. A proper approach is to let a (deterministic) depth-first search algorithm run through all the admissible paths and enumerate them. Theorem~\ref{mostlikelypaths} shows that if we know the total mean transition time along each path, then we can recover the mean transition time of the graph. 

In this case the \emph{computational complexity of the algorithm is factorial} in the number of nodes in $V$, since it depends in a delicate manner on the architecture of the graph. More precisely, the algorithm has a complexity of $\mathcal{O}(|U| |V|^2 |V|!)$.
\end{itemize}

See \cite{vdV19} for a deeper analysis of the algorithm complexity. 


\section{Proofs of the main theorems}
\label{sec:thproof}

The aim of this section is to prove the theorems in Section~\ref{sec:theorems}. In Section~\ref{subsec:prelim} we introduce some further definitions. In Section~\ref{subsec:prtwolem} we prove Lemmas~\ref{blockedU} and \ref{lemmavw}. In Section~\ref{subsec:nextnucl} we prepare for the proof of the main theorems (Propositions~\ref{meannextnucleationtime} and \ref{lawnextnucleationtime} below). In Sections~\ref{subsec:mlp}--\ref{subsec:law} we prove Theorems~\ref{mostlikelypaths}, \ref{meantransitiontime} and \ref{law}, respectively. Throughout the section, recall the notation discussed in Remark~\ref{conditionalrmk}.


\subsection{Preparatory results}
\label{subsec:prelim}

Consider an arbitrary bipartite graph $G= ((U,V),E)$ with $|V| = N$ and let $v_1, \dots, v_N$ be the nodes in $V$. The activation path that the network follows is denoted by $v_1^*,  \dots, v_N^*$, while the indices of the nodes that the algorithm picks are denoted by $Y_1, \dots, Y_N$ (as in Definition~\ref{defYk}). We want to study the transition time when the network follows an admissible path. When conditioning the network on a specific activation order, we can write
\begin{equation}
\{ Y_k = i \} = \{ v_k^* = v_i\},
\end{equation}
in the sense that saying that the $k$-th index $Y_k$ chosen by the algorithm equals $i$ is equivalent to saying that the $k$-th node $v_k^*$ activating in the network equals $v_i$.

\begin{definizione}{\bf [Iteration graph]}
For $k = 1, \dots, N$, suppose that $k-1$ nodes in $V$ have already been activated. Denote by $G_k=((V_k,U_k), E_k)$ the subgraph of $G=((U,V),E)$ consisting of:
\begin{itemize}
\item $V_k \subseteq V$, the set of nodes in $V$ that have not been activated yet, i.e., $V_k =  V \setminus \{\{v_{Y_i}\}_{0 < i < k}\}$.
\item $U_k \subseteq U$, the set of nodes in $U$ that are not neighbors of any of the nodes in $V$ that have already been activated, i.e., $U_k = U \setminus \bigcup_{0 < i < k} N(v_{Y_i})$ (recall \eqref{neighbors}).
\item $E_k \subseteq E$, the set of edges between $U_k$ and $V_k$, i.e., $E_k = \{ uv: u \in U_k, v \in V_k\}$.
\end{itemize}
Let $\bar{d}_k$ be the minimum degree of the nodes in $V_k$ and $n_k$ be the number of least degree forks in $G_k$.
\end{definizione}

\begin{definizione}{\bf [Minimum degree subset]}
\label{mindegsubset}
Define the set of nodes with minimum degree in $V$ as 
\begin{equation}
M(V) = \{ v' \in V\colon\, d(v') = \mathrm{min}_{v \in V} d(v) \}.
\end{equation}
\end{definizione}

\begin{lemma}{\bf [Probability of choosing the next node]} 
\label{lemmaYk}
Given the graph $G_k$, in the subcritical and the critical regime, the probability that the next node activating in $V_k$ is node $v_i$ is
\begin{equation}
\mathbb{P}(Y_k = i) = 
\begin{cases}
\frac{1}{n_k}, & \mathrm{if} \, \beta (\bar{d}_k -1) \leq 1,\, v_i \in M(V_k), \\ 
0, & \mathrm{if} \, \beta (\bar{d}_k -1) \leq 1,\, v_i \in V_k \setminus M(V_k), \\ 
\end{cases}
\end{equation}
which depends on the sequence of nodes already active in $V$. 
\end{lemma}
\begin{proof} 
By construction, the algorithm picks nodes in $M(V_k)$ before it picks nodes in $V_k \setminus M(V_k)$. It is therefore enough to count the number of forks of minimum degree at step $k$, which is $n_k$.
\end{proof}


\subsection{Proof: activation sticks and selects low degrees}
\label{subsec:prtwolem}

We next prove two lemmas from Section~\ref{sec:alg} that will be needed to prove Theorem~\ref{mostlikelypaths} in Section~\ref{subsec:mlp}.

\begin{proof}[Proof of Lemma~\ref{blockedU}]
Recall that in \eqref{aggress1} we assumed the activation rates to be $g_U(x) =Bx^{\beta}$ and $g_V(x) = B' x^{\beta'}$, with $B,B',\beta, \beta' \in (0, \infty)$ and $\beta' > \beta +1$. We claim that if a node $u \in U$ deactivates and one of its neighbors in $V$ activates at time $t_u$, then with high probability it will not activate anymore after time $t_u$. The moments when $u$ could possibly activate again are the moments when all its neighbors in $V$ are simultaneously inactive. We consider the worst case scenario when $u$ has only one active neighbor $v \in V$. Denote by $t_v$ the first moment when $v$ deactivates after $t_u$. This happens many times, since the activity period of a node is described by an exponential variable $Z$ with rate $1$. In contrast, the inactivity periods are very short, since the nodes in $V$ are very aggressive and the activation rates grow with the queue lengths, which tend to infinity as $r\to\infty$. We consider a time period of length equal to the total transition time, and we assume the transition time to be the longest possible (of order $r$). Then on average we have a number of possibilities for $u$ to activate that is bounded from above by 
\begin{equation}
\label{competitions}
\frac{\mathbb{E}[\mathcal{T}_G^{Q^0}]}{\mathbb{E}[Z]} = \mathbb{E}[\mathcal{T}_G^{Q^0}] 
\asymp r, \qquad r \to \infty,
\end{equation}
At each of these times, nodes $u$ and $v$ are both inactive and are competing with each other to activate again. Denote by $Z_u$ and $Z_v$ the lengths of the first inactivity periods of $u$ and $v$, respectively, after time $t_v$. Note that, since the queue lengths are slightly changing, $Z_u$ and $Z_v$ are not exactly exponentially distributed, but we can approximate them by comparing them with upper and lower-bounding exponentials via thinning and superposition of Poisson processes. 
We then have that with high probability node $v$ activates before node $u$, since
\begin{equation}
\label{probwinning}
\begin{aligned}
\mathbb{P}(Z_u < Z_v) &=  [1+o(1)]\,\frac{g_U(Q_u(t_v))}{g_U(Q_u(t_v)) + g_V(Q_v(t_v))}\\ 
&= [1+o(1)]\, \frac{Kr^{\beta}}{Kr^{\beta} + K'r^{\beta'}} = 
[1+o(1)]\, \frac{1}{(K'/K)r^{(\beta'-\beta)}} = o\bigg(\frac{1}{r}\bigg), \qquad r \to \infty,
\end{aligned}
\end{equation}
where we use that $\beta'>\beta+1$, and $K,K'$ are positive constants. After the first competition, the winner of subsequent competitions at times $t > t_v$ is determined by the minimum of two random variables (inactivity periods) with rates $g_U(Q_u(t))$ and $g_V(Q_v(t))$. Note that the queue lengths in $U$ are always of order $r$, except when we are in the supercritical regime. In this regime we are not interested in the competition between $u$ and $v$ anymore, since we know how long the transition takes. The queue lengths in $V$ start being of order $r$, increase while $u$ is active and decrease when $v$ is active, but remain always of order $r$. Indeed, in order for the queue lengths in $V$ to become $o(r)$ at least time $\frac{\gamma_V}{c- \rho_V} r \, [1+o(1)]$ must have passed, but the transition takes at most time $T_U(r) = \frac{\gamma_U}{c-\rho_U} r \, [1+o(1)]$ which is always smaller by \eqref{assumptionparameters}. 
Hence, in every subsequent competition $v$ activates first, with high probability, since the probability of $u$ winning is always $o \big(\frac{1}{r}\big)$. In the worst case scenario, nodes $u$ and $v$ compete with each other for the duration of the transition, hence, by \eqref{competitions}, order $r$ times. The probability of $u$ winning at least one competition is $r \, o\big(\frac{1}{r}\big) = o(1)$. Hence, with high probability, node $u$ will never win any competition against node $v$ and will remain blocked for the duration of the transition.
\end{proof}

\begin{proof}[Proof of Lemma~\ref{lemmavw}] 
We distinguish between $\beta (\bar{d}_k -1) < 1$ and $\beta (\bar{d}_k -1) = 1$.

\medskip\noindent
(I) $\beta (\bar{d}_k -1) < 1$. 
Recall from Theorem~\ref{teoremapaper1} that the expected time it takes for node $v$ to nucleate, which is asymptotically equal to the time it takes for its $\bar{d}_k$ neighbors in $U$ to become simultaneously inactive, is of order $r^{\beta(\bar{d}_k -1)}$ and satisfies
\begin{equation}
\label{eq1}
\lim_{r \to \infty} \mathbb{P}_{1_U} \bigg( \frac{\mathcal{T}_v^{Q^{k-1}}}{\mathbb{E}_{1_U}[\mathcal{T}_v^{Q^{k-1}}]} > x \bigg) 
= \mathcal{P}_{\mathrm{sub}}(x) = e^{-x}, \qquad x \in [0,\infty).
\end{equation}
Similarly, the expected time it takes for node $w$ to nucleate is asymptotically equal to the time it takes for its $d_k(w)$ neighbors in $U$ to become simultaneously inactive: it is of order $r^{\beta(d_k(w) -1)} \succ r^{\beta(\bar{d}_k -1)}$ if $\beta (d_k(w) -1) < 1$, and of order $r\succ r^{\beta(\bar{d}_k -1)}$ if $\beta (d_k(w) -1) \geq 1$. It also satisfies
\begin{equation}
\label{eq2}
\lim_{r \to \infty} \mathbb{P}_{1_U} \bigg( \frac{\mathcal{T}_w^{Q^{k-1}}}
{\mathbb{E}_{1_U}[\mathcal{T}_w^{Q^{k-1}}]} > x \bigg) 
= \mathcal{P}(x), \qquad x \in [0,\infty), 
\end{equation}
with $\mathcal{P}(x) \uparrow 1$ when $x \downarrow 0$ and $\mathcal{P}(x) \downarrow 0$ when $x \uparrow \infty$. Note that, in the limit as $r \to \infty$, the forks of $v$ and $w$ can be treated as independent of each other. Indeed, in Section~\ref{subsec:aindep} we proved that the nucleation time of a fork is asymptotically not influenced by the behavior of other forks sharing nodes with it.

For any given value of $t$, $\mathcal{T}_v^{Q^{k-1}} \leq t$ and $\mathcal{T}_w^{Q^{k-1}} \geq t$ imply $\mathcal{T}_v^{Q^{k-1}} \leq \mathcal{T}_w^{Q^{k-1}}$, i.e., $\mathcal{T}_w^{Q^{k-1}} < \mathcal{T}_v^{Q^{k-1}}$ implies $\mathcal{T}_v^{Q^{k-1}} > t$ or $\mathcal{T}_w^{Q^{k-1}} < t$, and thus
\begin{equation}
\mathbb{P}_{1_U} (\mathcal{T}_w^{Q^{k-1}} < \mathcal{T}_v^{Q^{k-1}}) \leq  \mathbb{P}_{1_U}(\mathcal{T}_v^{Q^{k-1}} > t) + \mathbb{P}_{1_U}(\mathcal{T}_w^{Q^{k-1}} < t).
\end{equation}
By \eqref{eq1}-\eqref{eq2}, when choosing $t=\sqrt{\mathbb{E}_{1_U}[\mathcal{T}_v^{Q^{k-1}}] \mathbb{E}_{1_U}[\mathcal{T}_w^{Q^{k-1}}]}$ and taking the limit $r \to \infty$, the right-hand side tends to 0, since $t \succ \mathbb{E}_{1_U}[\mathcal{T}_v^{Q^{k-1}}]$ and $t \prec\mathbb{E}_{1_U}[\mathcal{T}_w^{Q^{k-1}}]$.\\
(II) $\beta (\bar{d}_k -1) = 1$.
As before, we know the law of the nucleation time for the fork of $v$ (critical) and $w$ (supercritical). As shown in \cite{BdHNS18}, with high probability $\mathcal{T}_w^{Q^{k-1}}/\mathbb{E}_{1_U}[\mathcal{T}_w^{Q^{k-1}}]$ tends to $1$. Moreover, with high probability any nucleation time of a complete bipartite graph in the critical regime (including the fork of $v$) is smaller than the transition time of the same graph in the supercritical regime. 

\end{proof}


\subsection{Proof: most likely paths} 
\label{subsec:mlp}

\begin{proof}[Proof of Theorem~\ref{mostlikelypaths}]
We prove the three statements separately.

\medskip\noindent
(i) Assuming that the network does not follow the greedy algorithm is equivalent to assuming that at some step $k$ with $\beta (\bar{d}_k -1) \leq 1$ a node $w$ that does not have a minimum degree is chosen instead of a node $v$ with degree $\bar{d}_k$. The probability of a group of $d > \bar{d}_k$ nodes being simultaneously inactive before a group of $\bar{d}_k$ nodes is equivalent to the probability of activating $w$ before $v$, which satisfies
\begin{equation}
\lim_{r \to \infty} \mathbb{P}_{1_U} \big(\mathcal{T}_w^{Q^{k-1}} < \mathcal{T}_v^{Q^{k-1}}\big) = 0
\end{equation}
by Lemma~\ref{lemmavw}. Hence, with high probability the network activates nodes in $V$ in a greedy way, as described by the algorithm. By Lemma~\ref{blockedU}, we also know that the nodes in $U$ that have deactivated remain inactive for the duration of the transition process. Consequently, they do not influence any future activation attempt of the nodes in $V$, whose activation therefore follows the algorithm. In the supercritical regime, we are only interested in the order of activation of the nodes until the first supercritical node, for which the above reasoning still holds. 

\medskip\noindent
(ii) Note that the queue lengths $Q^k$ depend on the sequence of indices $(Y_1, \dots, Y_{k-1})$ describing the order of the activating nodes in $V$. Indeed, we have seen in Section~\ref{subsec:uql} that the queue lengths change according to which nodes have already been activated. Moreover, for $k > 1$, also the probabilities $\frac{1}{n_k}$ depend on the sequence $(Y_1, \dots, Y_{k-1})$. The reader should keep this in mind while going through the proof. The proof evolves in three steps.

\medskip\noindent
{\bf 1.}
Denote the graph $G=((U,V),E)$ by $G_1=((U_1,V_1),E_1)$. Write
\begin{equation} \label{first}
\mathbb{E}_{1_U}[\mathcal{T}_G^{Q^0}\mathbbm{1}_{\mathcal{E}(a^*)}] = \mathbb{E}_{1_U}[\mathcal{T}_{G_1}^{Q^0} \mathbbm{1}_{\mathcal{E}(a^*)}] = \sum_{i_1: \, v_{i_1} \in V_1}  \mathbb{E}_{1_U}[\mathcal{T}_{G_1}^{Q^0} \mathbbm{1}_{\mathcal{E}(a^*)} \,|\, Y_1 = i_1 ] \, \mathbb{P}(Y_1 = i_1).
\end{equation}
By Lemma~\ref{lemmaYk}, when $\beta (\bar{d}_1 -1) \leq 1$, not all the terms in the above sum have positive probability, but only the ones corresponding to forks of minimum degree $\bar{d}_1$ do (and they all have the same probability). Recall that this probability is $\frac{1}{n_1}$. Also recall that $\mathbb{E}_Q$ averages over the random values $Q^1, \dots, Q^{N-1}$ of the updated queue lengths. We can write the random variable $\mathcal{T}_{G_1}^{Q^0}$ as sum of three random variables 
\begin{equation}
\mathcal{T}_{G_1}^{Q^0} = \bar{\tau}_1 + R_{U_2}^1 + \mathbb{E}_Q[\mathcal{T}_{G_2}^{Q^1}],
\end{equation}
where $G_2 = ((U_2,V_2), E_2)$ with $U_2 = U_1 \setminus N(v_{Y_1})$, $V_2 = V_1 \setminus \{ v_{Y_1}\}$ and $E_2 = E_1 \setminus \{ (u,v): u \in N(v_{Y_1}) \}$, while $Q^1 = Q(\bar{\tau}_1 + R_{U_2}^1 )$. The first variable represents the time the network takes to switch the first node on, the second variable represents the time the network takes (after activating the first node) to reach the configuration with all the nodes in $U_2$ active (see Lemma~\ref{returntime}), while the third variable represents the transition time of the remaining graph when we take the first activating node out. Note that, by Corollary~\ref{pk}, if we condition the network to follow an admissible path with a specific first activating node, then we get
\begin{equation}
\mathbb{E}_{1_U}[\bar{\tau}_1 \,|\, Y_1 = i_1] = f_1 \,\mathbb{E}_{1_U}[\mathcal{T}_{v_{i_1}}^{Q^0}],
\end{equation}
where $f_1$ is the factor that arises from considering the minimum of random variables. Also the variable $\mathcal{T}_{G_2}^{Q^1}$ changes accordingly, but with an abuse of notation we may write it in the same way. Thus,
\begin{equation}
\begin{split}
\mathbb{E}_{1_U}\big[\mathcal{T}_{G_1}^{Q^0} \mathbbm{1}_{\mathcal{E}(a^*)} \,|\,  Y_1 = i_1\big] 
&= \mathbb{E}_{1_U}\big[(\bar{\tau}_1+ R_{U_2}^1 + \mathbb{E}_Q[\mathcal{T}_{G_2}^{Q^1}])\,|\,\mathcal{E}(a^*) \cap \{Y_1 = i_1 \}\big] \\
&= \mathbb{E}_{1_U}[\bar{\tau}_1\mathbbm{1}_{\mathcal{E}(a^*)}\,|\,Y_1 = i_1 ] 
+ o(1) + \mathbb{E}_{1_U}\big[\mathbb{E}_Q[\mathcal{T}_{G_2}^{Q^1}]\mathbbm{1}_{\mathcal{E}(a^*)}\,|\,Y_1 = i_1\big] \\
&= f_1 \,\mathbb{E}_{1_U}\big[\mathcal{T}_{v_{i_1}}^{Q^0}\mathbbm{1}_{\mathcal{E}(a^*)}\big] 
+ o(1) + \mathbb{E}_{1_U}\big[\mathbb{E}_Q[\mathcal{T}_{G_2}^{Q^1}]\mathbbm{1}_{\mathcal{E}(a^*)}\big], \qquad r \to \infty,
\end{split}
\end{equation}
and this holds with high probability due to Lemma~\ref{returntime}. We want to analyze the latter in a recursive way. The $k$-th iteration gives  
\begin{equation} \label{iter1}
\mathbb{E}_{1_U}\big[\mathbb{E}_Q[\mathcal{T}_{G_k}^{Q^{k-1}}]\mathbbm{1}_{\mathcal{E}(a^*)}\big] 
= \mathbb{E}_Q\Bigg[\sum_{i_k: \,v_{i_k} \in V_k}
\mathbb{E}_{1_U}\big[\mathcal{T}_{G_k}^{Q^{k-1}} \mathbbm{1}_{\mathcal{E}(a^*)}\,|\,Y_k = i_k\big] \,  \mathbb{P}(Y_k = i_k)\Bigg] .
\end{equation}

\medskip\noindent
{\bf 2.}
We can again write the random variable $\mathcal{T}_{G_k}^{Q^{k-1}}$ as sum of three random variables 
\begin{equation}
\mathcal{T}_{G_k}^{Q^{k-1}} = \bar{\tau}_k + R_{U_{k+1}}^k+ \mathbb{E}_Q[\mathcal{T}_{G_{k+1}}^{Q^k}],
\end{equation}
where $G_{k+1} = ((U_{k+1},V_{k+1}), E_{k+1})$ with $U_{k+1} = U_k \setminus N(v_{Y_k})$, $V_{k+1} = V_k \setminus \{  v_{Y_k}\}$ and $E_{k+1} = E_k \setminus \{ (u,v)\colon\, u \in N(v_{Y_k}) \}$, while $Q^k = Q(\sum_{j = 1}^{k} \bar{\tau}_j + R_{U_{j+1}}^j)$. By Corollary~\ref{pk}, we again have that 
\begin{equation}
\mathbb{E}_{1_U}[\bar{\tau}_k \,|\, Y_k = i_k] = f_k \,\mathbb{E}_{1_U}\big[\mathcal{T}_{v_{i_k}}^{Q^{k-1}}\big],
\end{equation}
and also the variable $\mathcal{T}_{G_{k+1}}^{Q^k}$ changes accordingly when it is conditioned (again, with an abuse of notation we write it in the same way). The inner conditional expectation in \eqref{iter1} can be written as 
\begin{equation} \label{iter2}
\begin{split}
\mathbb{E}_{1_U}\big[\mathcal{T}_{G_k}^{Q^{k-1}} \mathbbm{1}_{\mathcal{E}(a^*)}\,|\,Y_k = i_k\big] 
&= \mathbb{E}_{1_U}\big[(\bar{\tau}_k + R_{U_{k+1}}^k+ \mathbb{E}_Q[\mathcal{T}_{G_{k+1}}^{Q^k}])\mathbbm{1}_{\mathcal{E}(a^*)}\,|\,Y_k = i_k\big] \\
&= \mathbb{E}_{1_U}\big[\bar{\tau}_k\mathbbm{1}_{\mathcal{E}(a^*)}\,|\,Y_k = i_k\big] 
+ o(1) + \mathbb{E}_{1_U}\big[\mathbb{E}_Q[\mathcal{T}_{G_{k+1}}^{Q^k}]\mathbbm{1}_{\mathcal{E}(a^*)}\,|\,Y_k = i_k\big] \\
&=  f_k\, \mathbb{E}_{1_U}\big[\mathcal{T}_{v_{i_k}}^{Q^{k-1}}\mathbbm{1}_{\mathcal{E}(a^*)}] 
+ o(1) + \mathbb{E}_{1_U}[\mathbb{E}_Q[\mathcal{T}_{G_{k+1}}^{Q^k}]\mathbbm{1}_{\mathcal{E}(a^*)}\big], \qquad r \to \infty,
\end{split}
\end{equation}
which holds with high probability due to Lemma~\ref{returntime}.
At each iteration the conditional expectation reduces to a sum of three terms: the first term represents the expected time it takes to switch the following node on (adjusted by a factor that keeps track of the fact that the node activates before the other nodes), the second term represents the expected time the network takes (after activating the previous node) to reach the configuration with all the nodes remaining in $U$ active, while the third term represents the mean transition time of the remaining network when we take the following activating node out.

\medskip\noindent
{\bf 3.}
Note that, for each $k = 1, \dots, N$, the graph $G_{k+1}$ depends on the sequence of indices $(Y_1, \dots, Y_{k})$. Moreover, we know that also the queue lengths $Q^k$ depend on the indices $(Y_1, \dots, Y_{k-1})$. Thus, all the conditional expectations depend on the sequence of indices of activated nodes. Recall Definition~\ref{mindegsubset} and Lemma~\ref{lemmaYk}, by which the first iteration comes with a probability $\frac{1}{n_1}$ of choosing the first node activating, while each iteration with $k>1$ comes with a probability $\frac{1}{n_k}$, also depending on the sequence $(Y_1, \dots, Y_{k-1})$. After $k=2$ steps, using \eqref{iter1} and \eqref{iter2}, with high probability
\begin{equation}
\begin{split}
&\mathbb{E}_{1_U}\big[\mathcal{T}_{G_1}^{Q^0}\mathbbm{1}_{\mathcal{E}(a^*)}\big]\\ 
&= \sum_{i_1: \, v_{i_1} \in M(V_1)} \frac{1}{n_1} \, \mathbb{E}_{1_U}\big[\mathcal{T}_{G_1}^{Q^0} 
\mathbbm{1}_{\mathcal{E}(a^*)}\,|\,Y_1 = i_1\big] \\
&= \sum_{i_1: \, v_{i_1} \in M(V_1)} \frac{1}{n_1} \, \bigg(f_1\, \mathbb{E}_{1_U}\big[\mathcal{T}_{v_{i_1}}^{Q^0}
\mathbbm{1}_{\mathcal{E}(a^*)}\big] +o(1)
+ \mathbb{E}_{1_U}\big[\mathbb{E}_Q[\mathcal{T}_{G_2}^{Q^1}]\mathbbm{1}_{\mathcal{E}(a^*)}\big]\bigg) \\
&= \sum_{i_1: \, v_{i_1} \in M(V_1)} \frac{1}{n_1} \, \Bigg(f_1\, \mathbb{E}_{1_U}\big[\mathcal{T}_{v_{i_1}}^{Q^0}\mathbbm{1}_{\mathcal{E}(a^*)}\big] +o(1)+ \mathbb{E}_Q\Bigg[\sum_{i_2: \, v_{i_2} \in M(V_2)} \frac{1}{n_2} \, \mathbb{E}_{1_U}\big[\mathcal{T}_{G_2}^{Q^1}
 \mathbbm{1}_{\mathcal{E}(a^*)}\,|\,Y_2 = i_2\big] \Bigg] \Bigg) \\
&=  \sum_{i_1: \, v_{i_1} \in M(V_1)} \frac{1}{n_1} \, \Bigg( f_1\, \mathbb{E}_{1_U}\big[\mathcal{T}_{v_{i_1}}^{Q^0}\mathbbm{1}_{\mathcal{E}(a^*)}\big] +o(1) \\
&\qquad \,\,\,\,+ \mathbb{E}_Q\Bigg[\sum_{i_2: \, v_{i_2} \in M(V_2)} \frac{1}{n_2} \, \left(f_2\, 
\mathbb{E}_{1_U}\big[\mathcal{T}_{v_{i_2}}^{Q^1}\mathbbm{1}_{\mathcal{E}(a^*)}\big] 
+o(1) +\mathbb{E}_{1_U}\big[\mathbb{E}_Q[\mathcal{T}_{G_3}^{Q^2}]\mathbbm{1}_{\mathcal{E}(a^*)}\big]\right) \Bigg] \Bigg), \qquad r \to \infty.
\end{split}
\end{equation}
After $N$ steps we activate the last node in $V$, and the inner conditional expectation becomes
\begin{equation}
\begin{split}
\mathbb{E}_{1_U}\big[\mathcal{T}_{G_N}^{Q^{N-1}}\mathbbm{1}_{\mathcal{E}(a^*)}\big]  
&= \sum_{i_N: \, v_{i_N} \in M(V_N)} \frac{1}{n_N}\, \bigg( f_N\, \mathbb{E}_{1_U}\big[\mathcal{T}_{v_{i_N}}^{Q^{N-1}}\mathbbm{1}_{\mathcal{E}(a^*)}\big] + \mathbb{E}_{1_U} [R_{U_{N+1}}^N] + \mathbb{E}_{1_U}\big[\mathbb{E}_Q[\mathcal{T}_{G_{N+1}}^{Q^N}]\mathbbm{1}_{\mathcal{E}(a^*)}\big] \bigg) \\
& = \sum_{i_N: \, v_{i_N} \in M(V_N)} \frac{1}{n_N} \, f_N\,  \mathbb{E}_{1_U}\big[\mathcal{T}_{v_{i_N}}^{Q^{N-1}}\mathbbm{1}_{\mathcal{E}(a^*)}\big].
\end{split}
\end{equation}
Indeed, as soon as we activate the last node in $V$, we are actually done and we are not interested in what happens after. We can set $R_{U_{N+1}}^N = 0$ and we have $V_{N+1} = \emptyset$, which implies $\mathbb{E}_{1_U}\big[\mathbb{E}_Q[\mathcal{T}_{G_{N+1}}^{Q^N}]\mathbbm{1}_{\mathcal{E}(a^*)}\big] = 0$. Note that we are summing over sequences of nodes such that the $k$-th node is in $M(V_k)$. Therefore we are summing over the set $\mathcal{A}_k$ of (partial) admissible paths, and so we have arrived at \eqref{sumnucl}.

\medskip\noindent
(iii) The claim follows from analogous steps as in (ii), given any admissible path $a \in \mathcal{A}$.
\end{proof}


\subsection{Proof: mean of the transition time} 
\label{subsec:tr}

\begin{proof}[Proof of Theorem~\ref{meantransitiontime}]
Recall that, in the subcritical and the critical regime, we are computing the mean transition time conditioned on the event $A_a$ that the transition follows a fixed admissible path $a = (v_1, \dots, v_{N}) \in \mathcal{A}$. We again distinguish between the three regimes.

\medskip\noindent
(I) $\beta \in (0, \frac{1}{d^*-1})$: subcritical regime. Every term in the sum is of order $r^{\beta(d^*-1)} = o(r)$, which means that the significant terms are the ones with $\bar{d}_k = d^*$ only. The pre-factors of these terms are given by subcritical forks, and so, for any $u \in U$,
\begin{equation}
\begin{split}
\mathbb{E}_{1_U}[\mathcal{T}_G^{Q^0} \,|\, A_a]
&= \mathbb{E}_Q\left[\sum_{k: \,\bar{d}_k = d^*}  
f_k \, \mathbb{E}_{1_U}\big[\mathcal{T}_{v_k}^{Q^{k-1}}\big]\right]
= \mathbb{E}_Q\left[\sum_{k: \,\bar{d}_k = d^*} f_k \,  \frac{(Q_u^{k-1})^{\beta(d^*-1)}}{d^*B^{-(d^*-1)}}\right] [1+o(1)] \\
& = \sum_{k: \,\bar{d}_k = d^*} f_k \,\frac{\gamma_U^{\beta(d^*-1)}}{d^*B^{-(d^*-1)}}\, r^{\beta(d^*-1)} \, [1+o(1)], 
\qquad r \to \infty,
\end{split}
\end{equation}
with $f_k = \frac{1}{n_k}$. The last equality is obtained by using \eqref{gbeq} and \eqref{qlsubcritical}. 

\medskip\noindent
(II) $\beta = \frac{1}{d^*-1}$: critical regime. Every term in the sum is of order $o(r)$, except the terms with $\bar{d}_k = d^*$, which is of order $r$. The significant terms are the ones with $\bar{d}_k = d^*$ only. The pre-factors of these terms are given by critical forks, and so, for any $u \in U$,
\begin{equation}
\begin{split}
\mathbb{E}_{1_U}[\mathcal{T}_G^{Q^0} \,|\, A_a]
&= \mathbb{E}_Q\left[\sum_{k: \,\bar{d}_k = d^*}  f_k \, \mathbb{E}_{1_U}\big[\mathcal{T}_{v_k}^{Q^{k-1}}\big] \right]
=  \mathbb{E}_Q\left[\sum_{k: \,\bar{d}_k = d^*}  f_k \, \frac{Q_u^{k-1}}{d^*B^{-(d^*-1)}+c-\rho_U} \right] [1+o(1)] \\
&= \sum_{k: \,\bar{d}_k = d^*} f_k \, \frac{\gamma_U^{(k-1)}}{d^*B^{-(d^*-1)}+c-\rho_U}\, r\, [1+o(1)],
\end{split}
\end{equation}
with $\gamma_U^{(k-1)}$ defined in \eqref{qlcriticalrv} and
\begin{equation}
f_k = \frac{ \bar{d}_k B^{-(\bar{d}_k -1)} + c-\rho_U}{n_k \bar{d}_k B^{-(\bar{d}_k -1)} + c-\rho_U}.
\end{equation}
The last equality is obtained by using \eqref{gbeq} and \eqref{qlcritical}.

\medskip\noindent
(III) $\beta \in (\frac{1}{d^*-1}, \infty)$: supercritical regime. Denote by $v_{\textrm{sc}}$ the first supercritical node. We know from \eqref{qlsupercritical} that, after $v_{\textrm{sc}}$ is activated, the queue lengths become negligible (order $o(r)$), and the mean transition time is given by the expected time it takes for them to hit zero, i.e.,
\begin{equation}
\mathbb{E}_{1_U}\big[\mathcal{T}_{G}^{Q^0}\big] = T_U(r) = \frac{\gamma_U}{c-\rho_U}\, r\, [1+o(1)], \qquad r \to \infty.
\end{equation} 
\end{proof}


\subsection{Proof: law of the transition time} 
\label{subsec:law}

\begin{proof}[Proof of Theorem~\ref{law}]
We again distinguish between the three regimes.

\medskip\noindent
(I) $\beta \in (0, \frac{1}{d^*-1})$: subcritical regime. Recall that the significant terms in the sum for the mean transition time are those coming from nodes with degree $\bar{d}_k = d^*$ with $d^* < \frac{1}{\beta}+1$. There are $m_{\mathrm{sub}}^a$ such terms, where $m_{\mathrm{sub}}^a$ depends on the path $a \in \mathcal{A}$, and each term comes with a multiplicative factor $f_k$. We can write the transition time along path $a$ divided by its mean as
\begin{equation} \label{splittrtime}
\begin{split}
\frac{\mathcal{T}_G^{Q^0} \,|\, A_a}{\mathbb{E}_{1_U} [\mathcal{T}_G^{Q^0} \,|\, A_a]} 
& = \frac{\mathbb{E}_Q\left[\sum_{k=1}^N \bar{\tau}_k + \sum_{k=2}^{N} R_{U_k}^{k-1}\right]}{\mathbb{E}_{1_U} [\mathcal{T}_G^{Q^0} \,|\, A_a]} \\
& = \frac{ \mathbb{E}_Q\left[\sum_{k'\colon\, \bar{d}_{k'} = d^*} \bar{\tau}_{k'}+  \sum_{k''\colon\,\bar{d}_{k''} < d^*} 
\bar{\tau}_{k''} + \sum_{k=2}^{N} R_{U_k}^{k-1} \right]}{\mathbb{E}_{1_U} [\mathcal{T}_G^{Q^0} \,|\, A_a]}.
\end{split}
\end{equation}
We know that the law of a sum of independent random variables has a density given by the convolution of their densities. Here the nucleation times and the return times can be considered as independent, since they only depend on the queue lengths, which remain close to the initial value in the subcritical regime. 

There are three types of sums in the numerator of the last line of \eqref{splittrtime}. The first type of sum has terms of the form $\bar{\tau}_{k'}/\mathbb{E}_{1_U} [\mathcal{T}_G^{Q^0} \,|\, A_a]$, with $k'$ such that $\bar{d}_{k'} = d^*$. As $r \to \infty$, these are the significant terms, since they are of the same order as the mean transition time. For each of them, i.e., for each $k'$, we have
\begin{equation}
\begin{split}
\lim_{r \to \infty} \mathbb{P}_{1_U} \bigg(\frac{\bar{\tau}_{k'}}{\mathbb{E}_{1_U} [\mathcal{T}_G^{Q^0} \,|\, A_a]} > x\bigg) & = 
\lim_{r \to \infty} \mathbb{P}_{1_U} \bigg(\frac{\bar{\tau}_{k'}}{\mathbb{E}_{1_U} [\bar{\tau}_{k'}]} > \frac{ \mathbb{E}_{1_U} [\mathcal{T}_G^{Q^0} \,|\, A_a]}{\mathbb{E}_{1_U} [\bar{\tau}_{k'}]} \,x \bigg) \\
& = \exp\left(-\frac{\mathbb{E}_Q\left[\sum_{i\colon\,\bar{d}_{i}=d^*} \mathbb{E}_{1_U} [\bar{\tau}_{i}]\right]}{\mathbb{E}_{1_U} [\bar{\tau}_{k'}]}\,x\right) \\
& = \exp\left(-\frac{\sum_{i\colon\,\bar{d}_{i}=d^*}f_{i}}{f_{k'}}\,x\right), \qquad x \in [0, \infty),
\end{split}
\end{equation}
where in the second step we use Proposition~\ref{lawnextnucleationtime}. We write the density as 
\begin{equation}
\mathcal{P}_{\mathrm{sub}}^{f_{k'}, S_{\mathrm{sub}}^a}(x) =  \frac{S_{\mathrm{sub}}^a}{f_{k'}}\exp\left(-\frac{S_{\mathrm{sub}}^a}{f_{k'}}\,x\right), \qquad x \in [0, \infty),
\end{equation} 
with 
\begin{equation}
S_{\mathrm{sub}}^a = \sum_{i\colon\,\bar{d}_{i}=d^*}f_{i}.
\end{equation}
\noindent
The second type of sum has terms of the form $\bar{\tau}_{k''}/\mathbb{E}_{1_U} [\mathcal{T}_G^{Q^0} \,|\, A_a]$, with $k''$ such that $\bar{d}_{k''} < d^*$. As $r \to \infty$, these are negligible, since they are of smaller order than the mean transition time. For each of them, i.e., for each $k''$, we have
\begin{equation}
\lim_{r \to \infty} \mathbb{P}_{1_U} \bigg(\frac{\bar{\tau}_{k''}}{\mathbb{E}_{1_U} [\mathcal{T}_G^{Q^0} \,|\, A_a]} > x\bigg) = 
\lim_{r \to \infty} \mathbb{P}_{1_U} \bigg(\frac{\bar{\tau}_{k''}}{\mathbb{E}_{1_U} [\bar{\tau}_{k''}]} > \frac{ \mathbb{E}_{1_U} [\mathcal{T}_G^{Q^0} \,|\, A_a]}{\mathbb{E}_{1_U} [\bar{\tau}_{k''}]} \,x \bigg), \qquad x \in [0, \infty),
\end{equation}
and the density is $\delta_0$, the Dirac function at $0$. The third type is of the form $R_{U_k}^{k-1}/\mathbb{E}_{1_U} [\mathcal{T}_G^{Q^0} \,|\, A_a]$, with $k= 2, \dots, N$. As $r \to \infty$, these are also negligible, since they are $o(1)$ by Lemma~\ref{returntime}, and hence their density is also $\delta_0$.

The density of $(\mathcal{T}_G^{Q^0} | A_a)/\mathbb{E}_{1_U} [\mathcal{T}_G^{Q^0} |A_a]$ is given by the convolution of the densities of the three types of terms. Since $\delta_0$ gives the identity for the convolution, we can write 
\begin{equation}
\lim_{r \to \infty} \mathbb{P}_{1_U} \bigg( \frac{\mathcal{T}_G^{Q^0}}{\mathbb{E}_{1_U} [\mathcal{T}_G^{Q^0} | A_a]} > x  \,\big|\,  A_a \bigg) = \int_x^{\infty} \bigg(\circledast_{k' = 1}^{m_{\mathrm{sub}}^a} \mathcal{P}_{\mathrm{sub}}^{f_{k'}, S_{\mathrm{sub}}^a}\bigg) (y) dy, \qquad x \in [0, \infty),
\end{equation} 
and we can rename the index $k'$ by $k$. 

\medskip\noindent
(II) $\beta = \frac{1}{d^*-1}$: critical regime. For two reasons we do \emph{not} know how to handle this regime: (a) We do not know the law of the next nucleation times because Proposition~\ref{lawnextnucleationtime} only holds in the subcritical regime. (b) The next nucleation times are \emph{dependent} random variables, and so convolution is no longer relevant.

\medskip\noindent
(III) $\beta \in (\frac{1}{d^*-1}, \infty)$: supercritical regime. Recall that $T_U(r) =  \frac{\gamma_U}{c-\rho_U} r \, [1+o(1)]$, $r \to \infty$. The law of the transition time is given by $\mathcal{P}_3(x)$ from Theorem~\ref{teoremapaper1}. Indeed, the mean transition time is the expected time it takes for the queue lengths in $U$ to hit zero and. With high probability the transition does not occur before or after its mean. Since
\begin{equation}
\lim_{r \to \infty} \mathbb{P}_{1_U} \Big( \mathcal{T}_G^{Q^0} > \mathbb{E}_{1_U} \big[\mathcal{T}_G^{Q^0}\big] \Big) 
= \lim_{r \to \infty} \mathbb{P}_{1_U} \big( \mathcal{T}_G^{Q^0}> T_U(r) \big) = 0,
\end{equation}
we can write
\begin{equation} 
\lim_{r \to \infty} \mathbb{P}_{1_U} \bigg( \frac{\mathcal{T}_G^{Q^0}} {\mathbb{E}_{1_U} [\mathcal{T}_G^{Q^0}]} > x \bigg)
= 0, \qquad x \in [1, \infty).
\end{equation}
We also have
\begin{equation} 
\lim_{r \to \infty} \mathbb{P}_{1_U} \bigg( \frac{\mathcal{T}_G^{Q^0}} {\mathbb{E}_{1_U} [\mathcal{T}_G^{Q^0}]} > x \bigg) 
= 1, \qquad x \in [0,1).
\end{equation}
Hence the density is the Dirac function at $1$.
\end{proof}


\appendix
\renewcommand*{\thesection}{\Alph{section}}


\section{Appendix: minimum of independent forks}
\label{appA}

In this appendix we compute the mean next nucleation time in the situation where the forks competing for nucleation have no nodes in common, hence are independent of each other. Recall that, in the subcritical regime, the nucleation time of a fork is given by an exponential random variable, while in the critical regime it is given by a ``polynomial" random variable, in the sense that its law is truncated polynomial.


\subsection{Subcritical regime: exponential random variables}
\label{appA1}

Let $X_1, \dots, X_n$ be i.i.d.\ exponential random variables with rate $\lambda$. Let $Z = \min \{X_1, \dots, X_n \}$. Then
\begin{equation}
\mathbb{P}_{1_U}(Z > t) = \mathbb{P}_{1_U} (X_1 > t, \dots, X_n > t) = \mathbb{P}_{1_U} (X_1 > t)^n = e^{- n\lambda t}.
\end{equation}
Hence, $Z$ is an exponential random variable with rate $n \lambda$, and we have
\begin{equation}
\mathbb{E}_{1_U}[Z] = \frac{1}{n\lambda} = \frac{1}{n} \, \mathbb{E}_{1_U}[X_1].
\end{equation}
If we consider $X_1, \dots, X_{n_k}$ to be the nucleation times of independent forks of degree $\bar{d}_k$, and $Z$ to be the next nucleation time at step $k$, then we get  
\begin{equation}
\mathbb{E}_{1_U}[\bar{\tau}_k] = f_k^{\mathrm{iid}} \, \mathbb{E}_{1_U}\big[\mathcal{T}_{v_k^*}^{Q^{k-1}}\big], \qquad r \to \infty,
\end{equation}
with $f_k^{\mathrm{iid}} = \frac{1}{n_k}$. 


\subsection{Critical regime: polynomial random variables}
\label{appA2}

Let $X_1, \dots, X_n$ be i.i.d.\ polynomial random variables such that 
\begin{equation}
\mathbb{P}_{1_U}\bigg( \frac{X_i}{\mathbb{E}_{1_U}[X_i]} > x\bigg) = \left\{\begin{array}{ll}
(1-C x)^{\frac{1-C}{C}}, &\text{ if } x \in [0, \frac{1}{C}), \\[0.2cm]
0,  &\text{ if } x \in [\frac{1}{C}, \infty),
\end{array}
\right.
i = 1,\ldots,n,
\end{equation}
with
\begin{equation}
C = \frac{c-\rho_U}{\bar{d}_k B^{-(\bar{d}_k -1)} + c-\rho_U}.
\end{equation}
Let $Z = \min \{X_1, \dots, X_n \}$. Then, for $t = x \, \mathbb{E}_{1_U}[X_i]$, 
\begin{equation}
\mathbb{P}_{1_U}( X_i > t) = \left\{\begin{array}{ll}
\big(1-\frac{C}{\mathbb{E}_{1_U}[X_i]} t\big)^{\frac{1-C}{C}}, &\text{ if } t \in [0, \frac{\mathbb{E}_{1_U}[X_i]}{C}), \\[0.2cm]
0,  &\text{ if } t \in [\frac{\mathbb{E}_{1_U}[X_i]}{C}, \infty),
\end{array}
\right.
i = 1, \dots, n.
\end{equation}
Abbreviate $C = \frac{C_1}{C_1 + C_2}$, where $C_1 = c- \rho_U$ and $C_2 = \bar{d}_k B^{-(\bar{d}_k -1)}$. Then the exponent $\frac{1-C}{C}$ becomes $\frac{C_2}{C_1}$. We have
\begin{equation}
\begin{aligned}
&\mathbb{P}_{1_U}(Z > t) = \mathbb{P}_{1_U} (X_1 > t, \dots, X_n > t) = \mathbb{P}_{1_U} (X_1 > t)^n\\ 
&\qquad = \left\{\begin{array}{ll}
\big(1-\frac{C}{\mathbb{E}_{1_U}[X_i]} t\big)^{n\frac{C_2}{C_1}}, &\text{ if } t \in [0, \frac{\mathbb{E}_{1_U}[X_1]}{C}), \\[0.2cm]
0,  &\text{ if } t \in [\frac{\mathbb{E}_{1_U}[X_1]}{C}, \infty).
\end{array}
\right.
\end{aligned}
\end{equation}
The density function of $Z$ is 
\begin{equation}
f_z(t) = \frac{d}{dt}\big[1-\mathbb{P}_{1_U}(Z > t)\big] = \left\{\begin{array}{ll}
\frac{C}{\mathbb{E}_{1_U}[X_1]}\,n\,\frac{C_2}{C_1}\big(1-\frac{C}{\mathbb{E}_{1_U}[X_1]} t\big)^{n\frac{C_2}{C_1}-1}, 
&\text{ if } t \in [0, \frac{\mathbb{E}_{1_U}[X_1]}{C}), \\[0.2cm]
0,  &\text{ if } t \in [\frac{\mathbb{E}_{1_U}[X_1]}{C}, \infty).
\end{array}
\right.
\end{equation}
Hence
\begin{equation}
\mathbb{E}_{1_U}[Z] = \int_0^{\frac{\mathbb{E}_{1_U}[X_1]}{C}} f_Z(t) t \,dt = \frac{C}{\mathbb{E}_{1_U}[X_1]}n\frac{C_2}{C_1} \int_0^{\frac{\mathbb{E}_{1_U}[X_1]}{C}} \Big(1-\frac{C}{\mathbb{E}_{1_U}[X_1]} t\Big)^{n\frac{C_2}{C_1}-1} t \,dt.
\end{equation}
Substituting $u=1-\frac{C}{\mathbb{E}_{1_U}[X_1]} t$, we get
\begin{equation}
\begin{aligned}
\mathbb{E}_{1_U}[Z] &= \frac{\mathbb{E}_{1_U}[X_1]}{C}n\frac{C_2}{C_1} \int_0^{1} u^{n\frac{C_2}{C_1}-1}(1-u) \, du 
= \frac{\mathbb{E}_{1_U}[X_1]}{C}\,n\,\frac{C_2}{C_1} \bigg[ \int_0^1 u^{n\frac{C_2}{C_1}-1} \, du 
- \int_0^1 u^{n\frac{C_2}{C_1}} \, du \bigg] \\
&= \frac{\mathbb{E}_{1_U}[X_1]}{C}\,n\,\frac{C_2}{C_1} \bigg[\frac{1}{n\frac{C_2}{C_1}} - \frac{1}{n\frac{C_2}{C_1} +1} \bigg] 
= \frac{\mathbb{E}_{1_U}[X_1]}{C}\,n\,\frac{C_2}{C_1} \bigg[\frac{1}{n\frac{C_2}{C_1}(n\frac{C_2}{C_1} +1)} \bigg] 
= \frac{\mathbb{E}_{1_U}[X_1]}{C} \bigg[\frac{1}{n\frac{C_2}{C_1} +1} \bigg] \\
&= \mathbb{E}_{1_U}[X_1]  \frac{C_1 + C_2}{n C_2 + C_1}  
= \frac{\bar{d}_k B^{-(\bar{d}_k -1)} + c-\rho_U}{n \, \bar{d}_k B^{-(\bar{d}_k -1)} + c-\rho_U}\,  \mathbb{E}_{1_U}[X_1].
\end{aligned}
\end{equation}
If we consider $X_1, \dots, X_{n_k}$ to be the nucleation times of independent forks of degree $\bar{d}_k$, and $Z$ to be the next nucleation time at step $k$, then we get
\begin{equation}
\mathbb{E}_{1_U}[\bar{\tau}_k] =  f_k^{\mathrm{iid}}  \,\mathbb{E}_{1_U}[\mathcal{T}_{v_k^*}^{Q^{k-1}}], \qquad r \to \infty,
\end{equation}
with 
\begin{equation}
f_k^{\mathrm{iid}} =  \frac{\bar{d}_k B^{-(\bar{d}_k -1)} + c-\rho_U}{n_k \, \bar{d}_k B^{-(\bar{d}_k -1)} + c-\rho_U}.
\end{equation}



\end{document}